\documentclass[11pt]{article}
\usepackage[margin=1in]{geometry} 
\usepackage[utf8x]{inputenc}
\usepackage[OT4]{fontenc}
\usepackage[T1]{fontenc}
\usepackage{amssymb,amsfonts,amsmath,bbm,mathrsfs,stmaryrd}
\usepackage{enumerate}
\usepackage{tikz-cd}
\usepackage{tikz}
\usetikzlibrary{arrows,calc,decorations.pathreplacing,decorations.markings,intersections,shapes.geometric,through,fit,shapes.symbols,positioning,decorations.pathmorphing}

\usepackage[amsmath,thmmarks]{ntheorem}

\theoremstyle{plain}

\newtheorem{lemma}{Lemma}[section]
\newtheorem{proposition}[lemma]{Proposition}
\newtheorem{corollary}[lemma]{Corollary}
\newtheorem{theorem}[lemma]{Theorem}

\theoremstyle{nonumberplain}

\theoremstyle{plain}
\theorembodyfont{\upshape}
\theoremsymbol{\ensuremath{\blacklozenge}}

\newtheorem{definition}[lemma]{Definition}
\newtheorem{example}[lemma]{Example}
\newtheorem{remark}[lemma]{Remark}

\theoremstyle{nonumberplain}
\theoremsymbol{\ensuremath{\blacksquare}}

\newtheorem{proof}{Proof}

\newcommand\bC{{\mathbb C}}

\newcommand\bN{{\mathbb N}}

\newcommand\bZ{{\mathbb Z}}

\newcommand\cC{{\mathcal C}}

\newcommand\cK{{\mathcal K}}

\allowdisplaybreaks

\begin{document}

\title{Noncommutative numerable principal bundles\\ from group actions on C*-algebras}
\author{
  Mariusz Tobolski
}

\date{}

\newcommand{\Addresses}{{
  \bigskip
  \footnotesize

  \textsc{Instytut Matematyczny, Uniwersytet Wroc\l{}awski, pl. Grunwaldzki 2/4, 50-384 Poland}\par\nopagebreak \textit{E-mail address}:
  \texttt{mariusz.tobolski@math.uni.wroc.pl}
}}

\maketitle

\begin{abstract}
We introduce a definition of the locally trivial \mbox{$G$-C*-algebra}, which is a noncommutative counterpart of the total space of a~locally compact Hausdorff numerable principal \mbox{$G$-bundle}. To obtain this generalization, we have to go beyond the Gelfand--Naimark duality and use the multipliers of the Pedersen ideal. Our new concept enables us to investigate local triviality of noncommutative principal bundles coming from group actions on non-unital C*-algebras, which we illustrate through examples coming from $C_0(Y)$-algebras and graph C*-algebras. In the case of an action of a~compact Hausdorff group on a unital \mbox{C*-algebra}, local triviality in our sense is implied by the finiteness of the local-triviality dimension of the action. Furthermore, we prove that if $A$ is a~locally trivial $G$-C*-algebra, then the $G$-action on $A$ is free in a certain sense, which in many cases coincides with the known notions of freeness due to Rieffel and Ellwood.
\end{abstract}

\noindent {\em Key words:} noncommutative topology, C*-algebra, multipliers of the Pedersen ideal, C*-dynamical system, quantum principal bundle

\vspace{.5cm}

\noindent{MSC: 46L85, 37A55}



\section{Introduction}

Noncommutative mathematics (see e.g.~\cite{a-c94,fs-16}), motivated by the Heisenberg approach to quantum mechanics, is a way of generalizing certain classical mathematical theories which can be described in terms of commutative algebras, e.g. topology, measure theory, differential geometry, and group theory. In particular, the theory of locally compact Hausdorff spaces is (anti-)equivalent to the theory of commutative C*-algebras via the Gelfand--Naimark duality (see \S~\ref{nctop}). This leads to the philosophy of viewing noncommutative C*-algebras as algebras of functions on noncommutative spaces. The main idea of noncommutative topology is to transfer the important notions and results about locally compact Hausdorff spaces via the aforementioned equivalence and then check whether the assumption of commutativity can be dropped. In most cases, this process is far from obvious and it profits both sides of the equivalence. For instance, the recent breakthrough in the Elliott classification program of simple separable nuclear unital C*-algebras was made using the notion of a nuclear dimension~\cite{wz-10}, which generalizes the covering dimension of a locally compact Hausdorff space. 

The importance of principal bundles in topology and differential geometry can hardly be overestimated. They  also form the mathematical language of the classical gauge theory in physics (see e.g.~\cite{gl-n11}). The aim of the paper is to find and study the most general noncommutative analog of a principal bundle. This goal has already been achieved in the case when the spaces involved are compact Hausdorff. There are two main lines of research which focus on generalizing principal bundles to noncommutative mathematics:

\begin{enumerate}
\item[(I)]{\bf Compact principal bundles in noncommutative algebraic geometry}. In this setting, the role of a topological group and a topological space are played by a Hopf algebra and its unital comodule algebra respectively. The first definition of a~noncommutative principal bundle was given by Schneider~\cite{hj-s90} in terms of Hopf--Galois extensions~\cite{kt-81}. Brzezi\'nski and Majid~\cite{bm-93} used Schneider's definition and considered a~noncommutative gauge theory. The definition of a compact quantum principal bundle is now well established in the literature (see e.g.~\cite{alp-20,m-d96,pm-h96,hm-21,c-k04,lvs-05}). In fact, the notion that Schneider generalized was the one of a~principal $G$-bundle due to H.~Cartan~\cite{h-c49}, where the group action is free and proper but the local triviality (see e.g.~\cite{n-s51}) of the bundle is not assumed. The work of Aschieri, Fioresi, and Latini~\cite{afl-21} introduces a notion of local triviality which is suitable for noncommutative algebraic geomtery, but it cannot be used in the C*-algebraic context. Furthermore, the setting of Hopf--Galois extensions cannot be used to obtain a generalization of principal bundles beyond the compact case.

\item[(II)]{\bf Principal bundles in the theory of C*-algebras}. The interplay between principal bundles and C*-crossed products was studied by J.~Phillips and Raeburn in~\cite{pr-84} (see also \S~\ref{classlt}). The work of Connes and Rieffel~\cite{cr-87} was the first attempt to consider gauge theory, specifically the Yang--Mills theory, in the C*-algebraic framework~\cite{a-c94}.  Furthermore, in~\cite{ma-r90}, Rieffel defined free and proper actions of locally compact Hausdorff groups on C*-algebras generalizing the aforementioned notion of a principal bundle due to H.~Cartan. Some years later, Ellwood~\cite{da-e00} proposed a different generalization of a Cartan principal $G$-space in the context of actions of locally compact quantum groups on C*-algebras. 
In~\cite{dcy-13}, de Commer and Yamashita proved that, for compact (quantum) group actions on a unital C*-algebras, the definitions of Ellwood and Rieffel are equivalent. It is also worth mentioning that Echterhoff, Nest, and Oyono-Oyono introduced noncommutative principal torus bundles based on the Green theorem~\cite{enoo-09}. However, the notion of local triviality was still missing. In the case of a compact (quantum) group actions on a unital C*-algebras, a solution to the problem was proposed by the author, Gardella, Hajac, and Wu~\cite{ghtw-18}. Let us mention that the previous work of Budzy\'nski and Kondracki~\cite{bk-96}, while successful in many aspects, gave a generalization of a different concept, i.e. the piecewise triviality of the bundle (see e.g.~\cite{bhms-07}). This was observed by Hajac, Kr\"{a}hmer, Matthes, and Zieli\'nski in~\cite{hkmz-11} at the algebraic level.
\end{enumerate}

There are various interplays between the two approches. In the case of  compact quantum group actions on unital C*-algebras, Baum, de Commer, and Hajac proved in~\cite{bdch-17} that under additional assumptions the Hopf--Galois condition of Schneider is equivalent to the Ellwood condition. This equivalence is possible due to the existence of the Peter--Weyl functor.

In the paper, we introduce a definition of the {\em noncommutative numerable principal bundle} in the setting of actions of locally compact Hausdorff groups on C*-algebras. Unless otherwise stated, $G$ will denote a~locally compact Hausdorff group and $A$ will denote a C*-algebra. Recall that a principal bundle is called numerable if it admits a~trivializing cover with a partition of unity subordinated to it. The assumption of numerability is a crucial one. We have to rely on partitions of unity, since it is difficult to talk about open subsets of a noncommutative space. Recall, however, that only numerable principal bundles can be classified using the Milnor universal space~\cite{j-m56}. It is clear that already in topology it is important to go beyond the compact case. Indeed, although many principal $G$-bundles can be reduced to compact subgroups of $G$, there exist genuinely locally compact (i.e. non-reducible) ones. We use the terminology of~\cite[p.~332]{t-td08}, where total spaces of (numerable) principal $G$-bundles are called locally trivial $G$-spaces.

One way of obtaining a generalization of a principal bundle would be to use the Kasparov notion of a $C_0(Y)$-algebra~\cite{gg-k88}. We say that $A$ is a~{\em (numerable) principal $G$-$C_0(Y)$-algebra} if there is a~$G$-equivariant non-degenerate $*$-homomorphism $\chi:C_0(Y)\to ZM(A)$, where $Y$ is a locally compact Hausdorff $G$-space such that $Y\to Y/G$ is a (numerable) principal $G$-bundle (see \S~\ref{classlt}) and $ZM(A)$ denotes the center of the multiplier algebra of $A$. This definition covers all classical principal $G$-bundles with a locally compact Hausdorff total space and actions of dual groups on \mbox{C*-crossed} products coming from locally unitary actions considered in~\cite{pr-84}. However, this concept does not apply to simple \mbox{C*-al}\-ge\-bras. Since there are examples of simple C*-algebras giving rise to noncommutative principal bundles (see~\cite[Subsection~5.1]{cpt-21}), it seems that the notion of a principal $G$-$C_0(Y)$-algebra is more of a~classical flavour and a truly noncommutative object is needed.

The key difficulty in obtaining the right concept is that one has to go beyond the C*-algebra theory. This is prompted by several facts. First of all, the only known generalization of the local triviality condition for a principal $G$-bundle is the finiteness of the local-triviality dimension introduced in~\cite{ghtw-18}, which relies on the notion of a partition of unity. Therefore, we have to use unreduced cones or topological joins (see e.g.~\cite{j-m56}) one way or another. Let $\cC G$ denote the unreduced cone over $G$ (see \S~\ref{secone}). Observe that if $G$ is not compact then $\cC G$ is no longer locally compact. This means that $\cC G$ cannot be described using the Gelfand--Naimark duality. However, $\cC G$ is still a~compactly generated completely regular Hausdorff space. Hence, we can exploit the compact-open topology defined on the $*$-algebra $C(\cC G)$ of all continuous complex-valued functions on $\cC G$. In turn, this requires us to use the Pedersen multiplier algebra $\Gamma(K_A)$ (see \S~\ref{sec2.2}), which for $A=C_0(X)$, where $X$ is a locally compact Hausdorff space, is $*$-isomorphic to $C(X)$. Let us equip $C(\cC G)$ with the compact-open topology and $\Gamma(K_A)$ with the $\kappa$-topology (see~(\ref{kappatop})). We introduce the following definition (see Definition~\ref{fulllt}).

\vspace*{4mm}
{\bf The total space of a noncommutative numerable principal bundle}
\vspace*{2mm}

{\it We say that $A$ is a {\em locally trivial $G$-C*-algebra} if there exist unital $G$-equivariant continuous \mbox{$*$-ho}\-momorphisms $\rho_i:C(\cC G)\to\Gamma(K_A)$, $i\in\bN$ such that $(\sum_{i=0}^N\rho_i({\rm t}))_{N\in\bN}$ converges to $1$ in the \mbox{$\kappa$-topol}\-ogy.}
\vspace*{4mm}

If $G$ is compact and $A$ is unital, then the finiteness of the local triviality dimension immediately implies local triviality of the $G$-C*-algebra in question. We also have the following desired results. If $A=C_0(X)$ for some locally compact Hausdorff space $X$, then $A$ is a locally trivial $G$-C*-algebra if and only if $X\to X/G$ is a numerable principal $G$-bundle. Furthermore, numerable principal $C_0(Y)$-algebras are locally trivial as $G$-C*-algebras. Finally, if $A$ is a~locally trivial $G$-C*-algebra then the $G$-action on $A$ is {\em $\kappa$-free}. The $\kappa$-freeness is a generalization of freeness of the action of a group on a space (see \S~\ref{setthemfree}). If either $A$ is commutative or $G$ is compact and $A$ a unital, then $\kappa$-freeness is equivalent to the Ellwood freeness condition.

The paper is organized as follows. Section 2 contains some preliminaries on the Gelfand--Naimark duality and the multipliers of the Pedersen ideal, that will be needed in the sequel. We also describe the one-to-one correspondence between continuous maps of completely regular Hausdorff spaces and continuous $*$-homomorphisms of their unital $*$-algebras of all continuous complex-valued functions equipped with the compact-open topology. Results of a purely topological nature, including the reformulation of the notion of a numerable principal bundle, are in Section~3. Section 4 is the main part of the paper. We introduce three incarnations of local triviality of $G$-C*-algebras. We also consider classes of locally trivial $G$-C*-algebras coming from classical numerable principal bundles and principal numerable $G$-$C_0(Y)$-algebras. Finally, using the theory of graph C*-algebras, we provide examples of a non-unital locally trivial $\mathbb{Z}/2\mathbb{Z}$-C*-algebras that are not principal $C_0(X)$-algebras. 
Section 5 deals with freeness of actions. We introduce $\kappa$-freeness which is closely related to the Ellwood condition and prove that it is implied by local triviality.


\section{Preliminaries}\label{prelim}

We refer the reader to~\cite{gj-m90} for basic material on C*-algebras.

\subsection{The equivariant Gelfand--Naimark duality}\label{nctop}

Let $A$ be a C*-algebra and let $\widehat{A}$ be its {\em spectrum}, i.e. the set of unitary equivalence classes of irreducible $*$-representations of $A$. Let ${\rm Prim}(A)$ be the space of primitive ideals taken with the Jacobson topology. Recall that ${\rm Prim}(A)$ is always locally compact and that there is a surjective map $\widehat{A}\to {\rm Prim}(A)$, which enables us to turn $\widehat{A}$ into a locally compact space. If $\widehat{A}$ is a $T_0$-space, i.e. a space in which points are closed, then the aforementioned map is a homeomorphism. Furthermore, if $A$ is commutative, then $\widehat{A}$ is automatically Hausdorff.

An~{\em action} of a locally compact Hausdorff group $G$ on $A$ is given by a group homomorphism $\alpha:G\to {\rm Aut}(A)$. We say that $\alpha$ is {\em (jointly) continuous} if the map
$(a,g)\mapsto \alpha_g(a):=\alpha(g)(a)$ is continuous, where we consider the product topology on the domain.
A C*-algebra equipped with a continuous action of a~locally compact Hausdorff group $G$ is 
called a {\em $G$-C*-algebra}. Let $A$ and $B$ be $G$-C*-algebras with actions $\alpha$ and $\beta$ respectively. \mbox{A~$*$-homo}\-morphism $\Phi:A\to B$ is called {\em $G$-equivariant} or \mbox{a~{\em $G$-$*$-ho}\-{\em mo}}\-{\em mor}\-{\em phism} if $\Phi\circ\alpha_g=\beta_g\circ\Phi$ for all $g\in G$. A {\em morphism} between two C*-algebras $A$ and $B$ is a non-degenerate $*$-homomorphism $\phi:A\to M(B)$, i.e. $\phi(A)B$ is norm-dense in $B$. Here $M(B)$ is the C*-algebra of multipliers of $B$. Every action $\alpha$ of a locally compact Hausdorff group $G$ on a C*-algebra $A$ induces a continuous (right) action on $\widehat{A}$ by the formula
\begin{equation}\label{specact}
\widehat{A}\times G\longrightarrow\widehat{A}:\qquad ([\pi],g)\longmapsto [\pi\circ\alpha_g].
\end{equation}
The action $\alpha$ can also be extended to an action $M(\alpha)$ on $M(A)$ in the following way
\begin{equation}\label{multiaction}
M(\alpha)_g(x)a:=\alpha_g(xa),\qquad aM(\alpha)_g(x):=\alpha_g(ax),\qquad g\in G,\; x\in M(A),\; a\in A.
\end{equation}
The action $M(\alpha)$ is continuous with respect to the {\em strict topology} $\beta$, i.e. the topology generated by the seminorms
\begin{equation}\label{semistrict}
\phi_a(x):=\|xa\|,\quad \psi_a(x):=\|ax\|,\qquad x\in M(A),\quad a\in A.
\end{equation}
However, $M(\alpha)$ may not be continuous with respect to the norm topology. 
A {\em $G$-morphism} is a~$G$-equivariant morphism, where we consider the extension of the action to the multiplier algebra. In what follows, we will use the terms $G$-morphism and $G$-equivariant non-degenerate $*$-homomorphism interchangeably.

Finally, let $X$ be a locally compact Hausdorff space. Then the space $C_0(X)$ of all continuous complex-valued functions vanishing at infinity forms a commutative C*-algebra. If $G$ is a locally compact Hausdorff group acting on $X$ (on the right), then there is an induced continuous action on $C_0(X)$ given by
\begin{equation}\label{calgact}
(\alpha_g(f))(x):=f(xg),\qquad g\in G,\quad f\in C_0(X),\quad x\in X.
\end{equation}

C*-algebras are central objects studied in noncommutative topology which arose from the Gelfand--Naimark theorem~\cite{gn-43}. The equivariant version of the theorem can be stated as the following anti-equivalence of categories:
\begin{equation}\label{Ggelfnaim}
\left\{\begin{matrix}\text{compact Hausdorff $G$-spaces}\\
\text{and continuous $G$-maps}\end{matrix}\right\}
\quad{\longleftrightarrow}\quad
\left\{\begin{matrix}\text{unital commutative $G$-C*-algebras}\\
\text{and unital $G$-$*$-homomorphisms}\end{matrix}\right\}.
\end{equation}
Here by a compact Hausdorff $G$-space we mean that both the space and the group $G$ are compact Hausdorff. The above result can be extended to locally compact Hausdorff spaces in at least two different ways. Let us present the equivariant version of a generalization due to Woronowicz~\cite{sl-w80}:
\begin{equation}\label{GWor}
\left\{\begin{matrix}\text{locally compact Hausdorff $G$-spaces}\\
\text{and continuous $G$-maps}\end{matrix}\right\}
\quad{\longleftrightarrow}\quad
\left\{\begin{matrix}\text{commutative $G$-$C^*$-algebras}\\
\text{and $G$-morphisms}\end{matrix}\right\}.
\end{equation}
In the above anti-equivalences, the functor from left to right to each locally compact Hausdorff space $X$ assigns the commutative C*-algebra $C_0(X)$ and the inverse functor  to each commutative C*-algebra $A$ assigns its spectrum $\widehat{A}$. We omit the assignments of the morphisms for now, since they are analogous to~(\ref{funind}) and~(\ref{spind}) below.

\subsection{Beyond local compactness}\label{beyondGN}
In this paper, we have to go beyond the anti-equivalences~(\ref{Ggelfnaim}) and~(\ref{GWor}). Let $Y$ be a completely regular Hausdorff space and let $C(Y)$ denote the unital $*$-algebra of all continuous complex-valued functions on $Y$.  We equip $C(Y)$ with the compact-open topology $\kappa$, i.e. the topology generated by the seminorms
\begin{equation}\label{cosemi}
p_K(f):=\sup_{y\in K}|f(y)|,
\end{equation}
where $K\subset Y$ is compact and $f\in C(Y)$. This topology turns $C(Y)$ into a locally convex $*$-algebra. Let ${\rm Aut}_\kappa(C(Y))$ denote the group of $\kappa$-continuous  \mbox{$*$-auto}\-morphisms of $C(Y)$. Much in the same way as in~(\ref{calgact}), if $G$ is a locally compact Hausdorff group and $Y$ is a (right) $G$-space, we define a~$G$-action $\alpha:G\to{\rm Aut}_\kappa(C(Y))$ by the formula 
\begin{equation}\label{kappact}
(\alpha_g(f))(y):=f(yg),\qquad g\in G,\quad f\in C(Y), \quad y\in Y.
\end{equation}
One can show that this action is continuous, i.e. that the natural map $G\times C(Y)\to C(Y)$ it induces is continuous, where we consider the product topology on the domain. Let us define the set of {\em characters} 
\[
\Omega(C(Y)):=\{\chi:C(Y)\to\mathbb{C}~:~\chi\text{ is a non-zero $\kappa$-continuous $*$-homomorphism}\},
\]
which we can equip with the weak $*$-topology. The above \mbox{$G$-action} $\alpha:G\to{\rm Aut}_\kappa(C(Y))$ given by (\ref{kappact}) induces a $G$-action on $\Omega(C(Y))$ as follows
\begin{equation}\label{charact}
G\longrightarrow {\rm Homeo}(\Omega(C(Y))):\quad \chi\longmapsto \chi\circ\alpha_g\,.
\end{equation}
One can prove that the map
\begin{equation}\label{comregchar}
Y\longrightarrow\Omega(C(Y)),\qquad y\longmapsto m_y\,,\quad m_y(f):=f(y),
\end{equation}
is a continuous bijection, where complete regularity of $Y$ is only needed to prove injectivity~(see e.g.~\cite[Corollary~2.3]{mw-67}). Since the topology on a completely regular Hausdorff space $Y$ is the initial topology induced by $C(Y)$, we conclude that~(\ref{comregchar}) is a homeomorphism. Finally, since
\[
m_{yg}(f)=f(yg)=(\alpha_g(f))(y)=m_y(\alpha_g(f)),\qquad g\in G,\quad y\in Y,\quad f\in C(Y),
\]
we infer that~(\ref{comregchar}) is $G$-equivariant (this also proves that~(\ref{charact}) defines a continuous $G$-action).

If $F:Y\to Z$ is a continuous $G$-map between completely regular Hausdorff $G$-spaces, then we can define a unital $G$-equivariant $\kappa$-continuous $*$-homomorphism
\begin{equation}\label{funind}
F^*:C(Z)\longrightarrow C(Y):\qquad f\longmapsto f\circ F.
\end{equation}
Next, if we start with a unital $G$-equivariant $\kappa$-continuous $*$-homomorphism $F^*:C(Z)\to C(Y)$, then we can obtain a continuous $G$-map
\begin{equation}\label{spind}
F:\Omega(C(Y))\longrightarrow \Omega(C(Z)):\qquad \chi\longmapsto \chi\circ\Phi.
\end{equation}
Hence, using the $G$-homeomorphism~(\ref{comregchar}), we obtain a one-to-one correspondence between
\begin{itemize}
\item continuous $G$-maps $Y\to Z$ between completely regular Hausdorff $G$-spaces
\item and unital $G$-equivariant $\kappa$-continuous $*$-homomorphisms $C(Z)\to C(Y)$ between their locally convex $*$-algebras of all continuous functions.
\end{itemize}


\subsection{Multipliers of the Pedersen ideal}\label{sec2.2}

Let us recall the definition of the Pedersen ideal of a C*-algebra $A$~\cite{gk-p69}. We first introduce the subsets
\[
K_0^+:=\{a\in A_+~:~ab=a\text{ for some $b\in A_+$}\},
\]
\[
K^+:=\left\{a\in A_+~:~\text{there exists $a_0\,,\ldots, a_n\in K_0^+$ such that }a\leq \sum_{i=0}^na_n\right\}.
\]
The {\em Pedersen ideal} is defined to be $K_A:={\rm span}\,K^+$. It is the minimal dense ideal of $A$. 
A~{\em multipier} of $K_A$ is a pair of linear maps $(S,T)$ from $K_A$ to $K_A$ such that 
\[
xS(y)=T(x)y\qquad \text{for all $x,y\in K_A$}\,.
\] 
The set of all multipiers of $K_A$, denoted by $\Gamma(K_A)$, is a unital $*$-algebra with natural operations induced from $K_A$. Throughout the paper we call $\Gamma(K_A)$ the {\em Pedersen multiplier algebra} (or the algebra of multipliers of the Pedersen ideal) of $A$. The subset of bounded multipliers, i.e. when $S$ and $T$ are bounded maps, equals to the multiplier algebra $M(A)$. If $A=C_0(X)$ for some locally compact Hausdorff space $X$, then $K_A\cong C_c(X)$, the $*$-algebra of continuous compactly supported complex-valued functions on $X$, and $\Gamma(K_A)\cong C(X)$, the $*$-algebra of all continuous complex-valued functions on $X$. Observe that $\Gamma(K_A)$ is commutative if and only if $A$ is commutative. 

For the purposes of the sequel, let us discuss the $\kappa$-topology on $\Gamma(K_A)$~\cite[\S~3.6]{lt-76},
i.e. the locally convex topology given by the seminorms
\begin{equation}\label{kappatop}
\widetilde{\phi}_k((S,T)):=\|S(k)\|,\qquad \widetilde{\psi}_k((S,T)):=\|T(k)\|,\qquad k\in K_A\,.
\end{equation}
In this topology, $\Gamma(K_A)$ is a locally convex complete topological $*$-algebra~\cite[Theorem~3.8]{lt-76}. If $A=C_0(X)$ then the $\kappa$-topology on $\Gamma(K_A)\cong C(X)$ agrees with the compact-open topology~\cite[\S~4.1]{lt-76}.
The restriction of the $\kappa$-topology to $M(A)$ can be compared with the norm and the strict topology $\beta$ on it. In general, all these topologies are different, e.g. the norm and the strict topologies are complete, while $M(A)$ is only $\kappa$-dense in $\Gamma(K_A)$. 
However,~\cite[Theorem~1(iii)]{rc-b58} shows that in the commutative case the two topologies coincide on norm-bounded subsets of $M(A)$. Closely following the proof of this result, one obtains the following generalization:

\begin{lemma}\label{diftop}
Let $A$ be a C*-algebra. 
Then the strict topology $\beta$ coincides with the $\kappa$-topology on norm-bounded subsets of $M(A)$.\hfill$\blacksquare$
\end{lemma}

Let $A$ be a C*-algebra equipped with a continuous action $\alpha:G\to {\rm Aut}(A)$ of a locally compact Hausdorff group $G$. Since $\alpha_g(K_A)=K_A$ for every $g\in G$~(see e.g.~\cite[Theorem~1.4]{gk-p69}), we can define an action $\Gamma(\alpha)$ of $G$ on $\Gamma(K_A)$ by $\kappa$-continuous $*$-automorphisms as follows (see e.g.~\cite[\S~2.13]{lt-76})
\begin{equation}\label{gammact}
 \Gamma(\alpha)_g(S)(\alpha_g(k)):=\alpha_g(S(k)),\qquad\Gamma(\alpha)_g(T)(\alpha_g(k)):=\alpha_g(T(k)), 
\end{equation}
where $g\in G$, $(S,T)\in\Gamma(K_A)$, and $k\in K_A$. For $G$-C*-algebras $A$ and $B$ with actions $\alpha$ and $\beta$ respectively, when we say that a unital \mbox{$\kappa$-con}\-tin\-u\-ous $*$-homomorphism $\Gamma(K_A)\to\Gamma(K_B)$ is $G$-equivariant, we always mean the equivariance with respect to the extended actions $\Gamma(\alpha)$ and $\Gamma(\beta)$. Note that if $A$ is a~$G$-C*-algebra with an action $\alpha$, then the inclusion $M(A)\hookrightarrow\Gamma(K_A)$ is $G$-equivariant with respect to the extended actions $M(\alpha)$ and $\Gamma(\alpha)$ (see~(\ref{multiaction}) and~(\ref{gammact}) respectively).  In the commutative case, the $*$-isomorphism $C(X)\cong\Gamma(K_{C_0(X)})$ of topological $*$-algebras is $G$-equivariant if we consider the $G$-actions given by~(\ref{kappact}) and~(\ref{gammact}) respectively. 

Let us now turn to the relation between $G$-morphisms of C*-algebras and unital $\kappa$-continuous $G$-$*$-homomorphisms of their Pedersen multiplier algebras.
\begin{proposition}\label{onewayfunct}
Let $A$ and $B$ be $G$-C*-algebras, where $G$ is a locally compact Hausdorff group, and let $\Phi:\Gamma(K_A)\to\Gamma(K_B)$ be a unital \mbox{$G$-equi}\-variant $\kappa$-continuous $*$-homo\-morphism. Then the restriction of $\Phi$ to $A$ is a \mbox{$G$-mor}\-phism.
\end{proposition}
\begin{proof}
Let $\phi:=\Phi|_A$. Then $\phi$ is a $G$-equivariant $*$-homomorphism and $\phi(A)\subseteq M(B)$.
We only need to prove that it is non-degenerate, which is equivalent to showing that $\phi(e_\lambda)$ strictly converges to $1$ for some approximate unit $(e_\lambda)$ of $A$.
Observe that if $(e_\lambda)$ is an approximate unit of $A$, then $e_\lambda\to 1$ in the $\kappa$-topology in $\Gamma(K_A)$. Without loss of generality we may assume that $(e_\lambda)$ is bounded. We infer that
\[
\beta\lim_\lambda\phi(e_\lambda)=\kappa\lim_\lambda\phi(e_\lambda)=\kappa\lim_\lambda\Phi(e_\lambda)=\Phi(\kappa\lim_\lambda e_\lambda)=\Phi(1)=1,
\]
where in the first equality we used Lemma~\ref{diftop} (note that $\phi$ is bounded as a $*$-homomorphism between C*-algebras).
\end{proof}

It is not clear whether all $G$-morphisms can be extended to unital $G$-equivariant \mbox{$\kappa$-con}\-tin\-u\-ous \mbox{$*$-homo}\-morphisms. However, there is a class of morphisms for which this extension exist. A~morphism between C*-algebras $\phi:A\to M(B)$ is {\em strictly non-degenerate} if $K_B\subseteq \phi(K_A)B$~(see \cite[Definition~3.1]{m-a04}). 
First, let us identify a class of morphisms which are automatically strictly non-degenerate. In what follows, $ZA$ denotes the center of a C*-algebra $A$.
\begin{proposition}\label{pedcom}
Let $\phi:A\to M(B)$ be a morphism. If $\phi(A)\subseteq ZM(B)$, then $\phi$ is strictly non-degenerate.
\end{proposition}
\begin{proof}
First, note that $B\phi(A)B$ is a two-sided ideal of $B$ which is dense by non-degeneracy of $\phi$. By minimality of the Pedersen ideal, we obtain that $K_B\subseteq B\phi(A)B$. Since $\phi(A)\subseteq ZM(B)$ and $B^2=B$, we conclude that $\phi$ is strictly non-degenerate.
\end{proof}
Next, we show that strictly non-degenerate $G$-morphisms extend to unital $G$-equivariant $\kappa$-continuous $*$-homo\-morphisms.
\begin{proposition}\label{stricnondeg}
Let $\phi:A\to M(B)$ be a strictly non-degenerate $G$-morphism. Then $\phi$ extends uniquely to a unital $G$-equivariant $\kappa$-continuous $*$-homomorphism $\Phi:\Gamma(K_A)\to\Gamma(K_B)$.
\end{proposition}
\begin{proof}
We follow~\cite[p.~21]{m-a04} and define the map $\Phi:\Gamma(K_A)\to\Gamma(K_B)$ by
\[
\Phi(S)(\phi(k)b):=\phi(S(k))b,\qquad \Phi(T)(b\phi(k)):=b\phi(T(k)),\qquad k\in K_A,\quad b\in B.
\]
Here we use that $K_B\subseteq\phi(K_A)B=B\phi(K_A)$. It is plain to show that $(\Phi(S),\Phi(T))\in\Gamma(K_B)$.
Then it is equally straightforward to show that $\Phi$ is a unital $*$-homomorphism. Its uniqueness follows for example from the uniqueness of the extension to the multiplier algebras $M(A)\to M(B)$. We need to show the $\kappa$-continuity of $\Phi$. Let $(S_\lambda\,,T_\lambda)\to (S,T)$ in the $\kappa$-topology and let
\[
(U_\lambda\,,V_\lambda):=\Phi((S_\lambda\,,T_\lambda)),\qquad (U,V):=\Phi((S,T)).
\]
Since for every $k_B\in K_B$\, there exist $k_A, k'_A\in K_A$ and $b,b'\in B$ such that $k_B=\phi(k_A)b=b'\phi(k'_A)$, we obtain that
\[
\lim_\lambda U_\lambda(k_B)=\lim_\lambda U_\lambda(\phi(k_A)b)=\lim_\lambda \phi(S_\lambda(k_A))b= \phi(S(k_A))b=U(\phi(k_A)b)=U(k_B),
\]
\[
\lim_\lambda V_\lambda(k_B)=\lim_\lambda V_\lambda(b'\phi(k'_A))=\lim_\lambda b'\phi(T_\lambda(k'_A))= b'\phi(T(k'_A))=V(b'\phi(k'_A))=V(k_B).
\]
Assume that the $G$-actions on $A$ and $B$ are denoted by $\alpha$ and $\beta$ respectively. For all $g\in G$, $(S,T)\in\Gamma(K_A)$, $k\in K_A$, and $b\in B$, we have that
\begin{align*}
\Gamma(\beta)_g(\Phi(S))(\phi(\alpha_g(k))\beta_g(b))&=\Gamma(\beta)_g(\Phi(S))(\beta_g(\phi(k)b))=\beta_g(\Phi(S)(\phi(k)b))
\\&=\beta_g(\phi(S(k))b)=\phi(\alpha_g(S(k)))\beta_g(b)
\\&=\phi(\Gamma(\alpha)_g(S)(\alpha_g(k)))\beta_g(b)=\Phi(\Gamma(\alpha)_g(S))(\phi(\alpha_g(k))\beta_g(b))
\end{align*}
and the computation is the same for $\Gamma(\beta)_g(\Phi(T))$. Therefore, the $*$-homomorphism $\Phi$ is $G$-equi\-variant.
\end{proof}


\section{Topological aspects}

In this section, we discuss various topological subtleties that arise when one tries to find a right analog of a principal bundle in the theory of group actions on C*-algebras. 

\subsection{Unreduced cones of locally compact Hausdorff groups}\label{secone}

Although some of the concepts and results of this subsection apply to general topological spaces, we focus on the case of a locally compact Hausdorff group $G$. The {\em unreduced cone} $\cC G$ of $G$ is the quotient topological space
\[
\cC G:=([0,1]\times G)/(\{0\}\times G).
\]
We denote the quotient topology on $\cC G$ by $\tau_q$, the quotient map by $\pi:[0,1]\times G\to\cC G$, and call $[(0,g)]\in\cC G$ the {\em tip} of the cone. There is a continuous (right) $G$-action on $\cC G$ defined as follows
\[
\cC G\times G\longrightarrow \cC G:\qquad ([(t,g)],h)\longmapsto [(t,gh)].
\]
Indeed, since $G$ is locally compact Hausdorff, this action is continuous by an old lemma of Whitehead~\cite[Lemma~4]{jhc-w48} (see~\cite{e-m68} for a modern formulation of the result).

Recall that every locally compact Hausdorff group is paracompact Hausdorff~(see e.g.~\cite[Exercise~9, p.~261]{jr-m00} or~\cite[Corollary~3.1.4]{at-08} for a stronger statement). Then $[0,1]\times G$ is paracompact Hausdorff as a product of a compact Hausdorff space and a paracompact Hausdorff space.
Taking this into account, it is straightforward to show that the space $\cC G$ is completely regular Hausdorff. However, it may no longer be locally compact (the tip in $\cC G$ does not have any compact neighborhood). These observations are crucial from the point of view of noncommutative topology, since it implies the anti-equivalences~(\ref{Ggelfnaim}) and~(\ref{GWor}) do not apply in this case.

Consider the maps
\begin{equation}\label{height}
{\rm t}:\cC G\to [0,1],\qquad [(t,g)]\mapsto t,
\end{equation}
\vspace*{-3mm}
\begin{equation}\label{gfunct}
{\rm g}:{\rm t}^{-1}((0,1])\to G,\qquad [(t,g)]\mapsto g.
\end{equation}
We call ${\rm t}$ the {\em height} function. The above maps are continuous with respect to $\tau_q$. Indeed, observe that ${\rm t}\circ\pi:[0,1]\times G\to [0,1]$ is the projection onto the first factor. Similarly, since ${\rm t}^{-1}((0,1])\cong(0,1]\times G$, the map ${\rm g}\circ\pi|_{(0,1]\times G}$ is the projection onto the second factor.
For the purposes of the next subsection, we consider a different topology $\tau_M$ on $\cC G$, i.e. the initial topology with respect to the maps~(\ref{height}) and~(\ref{gfunct}) with a subbase given by the collection
\[
\{{\rm t}^{-1}(U),\; {\rm g}^{-1}(V)~:~U\text{ is open in }[0,1],\; V\text{ is open in }G\}.
\]
The above topology was introduced by Milnor for the topological join in~\cite{j-m56} (observe that $\cC G$ as a set is equal to the topological join of $G$ with the singleton set). 
Note that $\tau_q$ is finer than $\tau_M$. Furthermore, the map $\pi:[0,1]\times G\to\cC G$ is continuous with respect to the topology $\tau_M$ on $\cC G$. Our aim is to show that these topologies are in fact the same for every locally compact Hausdorff group $G$.
\begin{lemma}\label{fcg}
Let $G$ be a locally compact Hausdorff group. Then the family of compact subsets
\[
\mathcal{F}_{\cC G}:=\{\pi(K)~:~K\text{ is a compact subset of $[0,1]\times G$}\}
\]
generates the topology on $(\cC G,\tau_M)$.
\end{lemma}
\begin{proof}
We need to prove that
\[
A\subseteq \cC G\text{ is closed}\quad\iff\quad A\cap B\text{ is closed in }B\text{ for all }B\in\mathcal{F}_{\cC G}.
\]
($\Rightarrow$) Assume that $A$ is closed and take $B\in\mathcal{F}_{\cC G}$. The set $B$ is of the form $\pi(K)$, where $K$ is a~compact subset of $[0,1]\times G$. Since $\pi$ is continous with respect to $\tau_M$, $B$ is closed as a compact subset of a~Hausdorff space. Hence, $A\cap B$ is closed.

\noindent($\Leftarrow$) It suffices to show that any limit point of $A$ in $\cC G$ is a limit point of $A\cap\pi(K)$ for some compact $K\subseteq [0,1]\times G$. First, assume that $A$ has a~limit point $[(t,g)]\in\cC G\setminus\{[(0,g)]\}$. Since the continuous map $\pi$ is a homeomorphism when restricted to $(0,1]\times G$, the result follows by taking $K$ to be any compact neighbourhood of $(t,g)$ in the locally compact product space $(0,1]\times G$. Next, assume that the tip $[(0,g)]$ is a limit point of $A$, i.e. that there is a  net $(x_\lambda)_{\lambda\in\Lambda}$ in $A$ that converges to $[(0,g)]$ in $(\cC G,\tau_M)$ (without loss of generality we assume that this net is not  eventually constant). By local compactness, pick any $g\in G$ and its open neighbourhood $V$ such that its closure $\overline{V}$ is compact. Consider the following open neighbourhood
\[
U:=\cC G\setminus {\rm g}^{-1}(G\setminus V) 
\]
of $[(0,g)]$ in $(\cC G,\tau_M)$. Then there exists $\lambda_0\in\Lambda$ such that for all $\mu\geq\lambda_0$ we have that $x_\mu\in U$. The non-empty set
\[
\Lambda_U:=\{\mu\in\Lambda~:~\mu\geq\lambda_0\}
\]
is cofinal, so we can consider a subnet $(x_\mu)_{\mu\in\Lambda_U}$\, that has to converge to $[(0,g)]$ as well. Take $K:=[0,1]\times\overline{V}$. Since $x_\mu\in\pi(K)$ for all $\mu\in\Lambda_U$, we obtain that $[(0,g)]$ is a limit point of $A\cap\pi(K)$.
\end{proof}
\begin{remark}
Note that the above lemma is true for any locally compact Hausdorff space.
\end{remark}

\begin{proposition}\label{topequiv}
Let $G$ be a locally compact Hausdorff group. Then the topologies $\tau_q$ and $\tau_M$ on $\cC G$ are equivalent.
\end{proposition}
\begin{proof}
By Lemma~\ref{fcg}, the topology on $(\cC G,\tau_M)$ is generated by compact subsets of the form $\pi(K)$, where $K$ is a compact subset of $[0,1]\times G$. Let us observe that the same is true for $(\cC G,\tau_q)$. 
First, note that $(\cC G,\tau_q)$ is compactly generated as a Hausdorff quotient of a compactly generated space (see e.g.~\cite[\S~2.6]{ne-s67}).
Next, since $G$ is paracompact Hausdorff, from~\cite[Theorem~9]{r-g03} we infer that $\pi$ is semiproper (or compact-covering), i.e. for every compact subset $B$ of $(\cC G,\tau_q)$, there exists a~compact subset $K$ of $[0,1]\times G$ such that $B\subseteq\pi(K)$. This concludes the proof.
\end{proof}
\begin{remark}
Note that by~\cite[Theorems~5~and~6]{r-g03} the above result might not be true for arbitrary locally compact Hausdorff spaces. However, it will be true for any locally compact Hausdorff space that is paracompact.
\end{remark}

In conclusion, we have proved that for a locally compact Hausdorff group $G$, the topology on $\cC G$ is both an initial topology with respect to the maps~${\rm t}$ and~${\rm g}$ and a final topology with respect to the map $\pi$. This means that we are well placed for deciding continuity of functions going both from and into the unreduced cone of $G$.


\subsection{Numerable principal bundles}

Let $X$ be a topological space and let $\{f_\lambda\}$ be a family of continuous functions from $X$ to $[0,1]$ indexed by a set $\Lambda$. Recall that $\{f_\lambda\}$ is a {\em partition of unity} on $X$ if 
\begin{enumerate}
\item it is {\em locally finite}, i.e. for every $x\in X$ there is a neighbourhood $U$ of $x$ such that $f_\lambda|_{U}\neq 0$ for only finitely many $\lambda\in\Lambda$,
\item and the finite sum $\sum_{\lambda\in\Lambda} f_\lambda(x)$ equals $1$ for all $x\in X$. 
\end{enumerate}
A covering $\{U_\lambda\}$ of $X$ is called {\em numerable} if there is a partition of unity $\{f_\lambda\}$ on $X$ {\em subordinated} to it, i.e. ${\rm supp}f_\lambda\subseteq U_\lambda$ for all $\lambda\in\Lambda$, where ${\rm supp}f_\lambda$ denotes the closed support of $f_\lambda$. By~\cite[Corollary~13.1.9]{t-td08} we can always reduce arbitrary partitions of unity to countable ones. Hence, in what follows we assume that $\Lambda=\bN$.  

Following the terminology of~\cite[p.~332]{t-td08}, we call a topological $G$-space $X$ {\em trivial} if there is a~$G$-map $X\to G$ and {\em locally trivial} if it has an open covering by trivial $G$-spaces (such a cover is called {\em trivializing}). If $X$ is a locally trivial $G$-space, we say that $X\to X/G$ is a {\em principal $G$-bundle}. 
A~principal bundle $X\to X/G$ is called {\em numerable} if there is a numerable trivializing cover of $X$. In particular, when $X$ is paracompact Hausdorff, then any principal $G$-bundle $X\to X/G$ is numerable. 

Let us now give a new description of numerable principal bundles, which admits a noncommutative generalization (see \S~\ref{mainsubsection}).
\begin{theorem}\label{numprin}
Let $G$ be a locally compact Hausdorff group and let $X$ be a locally compact Hausdorff \mbox{$G$-space}. Then the following conditions are equivalent:
\begin{enumerate}
\item[{\rm (1)}] $X\to X/G$ is a numerable principal $G$-bundle.
\item[{\rm (2)}] There exist $G$-equivariant continuous maps $\rho_i:X\to\cC G$, $i\in\bN$, such that 
the sequence of continuous functions $(\sum_{i=0}^N({\rm t}\circ\rho_i))_{N\in\bN}$ converges to $1$ in the compact-open topology.
\end{enumerate}
\end{theorem}
\begin{proof}
First, assume that $X\to X/G$ is a numerable principal $G$-bundle. Let $\{U_i\}_{i\in\bN}$ be a~numerable trivializing cover of $X$ with trivializing maps $\varphi_i:U_i\to G$ and let $\{f_i\}$ be the $G$-invariant partition of unity  subordinated to that cover. One can always find such a partition of unity since every trivializing cover is $G$-invariant. Let $\widetilde{\varphi}_i$ be any (not necessarily continuous) extension of $\varphi_i$ to whole $X$.
Next, we define the maps
\[
\rho_i:X\overset{f_i\times\widetilde{\varphi}_i}{\longrightarrow}[0,1]\times G\overset{\pi}{\longrightarrow} \cC G:\quad x\longmapsto  [(f_i(x),\widetilde{\varphi}_i(x))] .
\]
Since $({\rm t}\circ\rho_i)(x)=f_i(x)$ and $({\rm g}\circ\rho_i)(x)=\varphi_i(x)$ (whenever it is well defined), from the definition of the Milnor topology $\tau_M$ and Proposition~\ref{topequiv}, we infer that the maps $\rho_i$ are continuous. They are also $G$-equivariant as each $\varphi_i$ is $G$-equivariant and the tip $[(0,g)]$ is a fixed point of the $G$-action on $\mathcal{C}G$. Observe that
\[
\sum_{i=0}^N({\rm t}\circ\rho_i)=\sum_{i=0}^Nf_i\qquad\text{for all $N\in\bN$}.
\]
Since $\{f_i\}$ is a partition of unity, the sequence $(\sum_{i=0}^N({\rm t}\circ\rho_i))_N$ converges locally uniformly to $1$. Hence, by local compactness of $X$, it also converges to $1$ in the compact-open topology.

Next, assume that there are $G$-equivariant maps $\rho_i:X\to \cC G$ such that $(\sum_{i=0}^N({\rm t}\circ\rho_i))_N$ converges  to $1$ in the compact-open topology. Compact convergence implies pointwise convergence, so that the family $\{{\rm t}\circ\rho_i\}$ is a~$G$-invariant generalized partition of unity in the sense of~\cite[p.~320]{t-td08}. By~\cite[Lemma~13.1.7]{t-td08} the covering $\{U_i:=({\rm t}\circ\rho_i)^{-1}((0,1])\}$ is numerable. Furthermore, the equivariance of the maps $\rho_i$ implies that this cover is $G$-invariant. It is left to check whether it is trivializing. To this end, we define the maps
\[
\varphi_i:U_i\longrightarrow G:\qquad x\longmapsto {\rm g}(\rho_i(x)).
\]
Here ${\rm g}:{\rm t}^{-1}((0,1])\to G$ is the map given by~(\ref{gfunct}). Each $\varphi_i$ is continuous and $G$-equivariant as a~composition of two continuous $G$-equivariant maps, so we are done.
\end{proof}
\begin{remark}
We formulate Theorem~\ref{numprin} using the compact-open topology, since this topology can be generalized to the noncommutative setting. However, the statement of the theorem would remain true if we would use locally uniform or even point-wise convergence instead.
\end{remark}


\section{Noncommutative numerable principal bundles}\label{nnpb}

In this section, we present a definition which generalizes the notion of a locally compact Hausdorff numerable principal bundle to the noncommutative setting. At the same time, this definition will generalize the local-triviality dimensions to the context of actions of locally compact Hausdorff groups on (not necessarily unital) C*-algebras. 

\subsection{Locally trivial $G$-C*-algebras}\label{mainsubsection}
We start with some basics on the local-triviality dimensions. Recall from \S~\ref{secone} that ${\rm t}$ denotes the height function in $C(\cC G)$. The following definition should be compared with~\cite{cpt-21,ghtw-18}.
\begin{definition}\label{ltdim}
Let $G$ be a compact Hausdorff group and let $A$ be a unital $G$-C*-algebra.
\begin{enumerate}
\item The {\em weak local-triviality dimension} $\dim_{\rm WLT}^G(A)$ is the smallest $n$ for which there exist unital $G$-equivariant $*$-homomorphisms $\rho_0\,,\ldots,\rho_n:C(\cC G)\to A$ such that $\sum_{i=0}^n\rho_i({\rm t})$ is invertible.
\item The {\em (plain) local-triviality dimension} $\dim_{\rm LT}^G(A)$ is the smallest $n$ for which there exist unital $G$-equivariant $*$-homomorphisms $\rho_0\,,\ldots,\rho_n:C(\cC G)\to A$ such that $\sum_{i=0}^n\rho_i({\rm t})=1$.
\item The {\em strong local-triviality dimension} $\dim_{\rm SLT}^G(A)$ is the smallest $n$ for which there exist unital $G$-equivariant $*$-homomorphisms $\rho_0\,,\ldots,\rho_n:C(\cC G)\to A$ with commuting images and such that $\sum_{i=0}^n\rho_i({\rm t})$=1.
\end{enumerate}
In any of the above cases, if there is no such finite $n$, then the associated dimension equals $\infty$.
\end{definition}
We list some known results about the local-triviality dimensions (see~\cite{cpt-21,ghtw-18}):
\begin{enumerate}
\item[(i)] We always have that
\[
\dim_{\rm SLT}^G(A)\geq \dim_{\rm LT}^G(A)\geq \dim_{\rm WLT}^G(A).
\]
When $A$ is commutative, then the above inequalities become equalities. There are also examples showing that these inequalities can be strict.
\item[(ii)] If a compact Hausdorff group $G$ acts on a unital C*-algebra of the form $A=C(X)$ for some compact Hausdorff space $X$, then finiteness of any of the local-triviality dimensions is equivalent to $X\to X/G$ being a principal $G$-bundle.
\item[(iii)] If $\dim_{\rm WLT}^G(A)<\infty$, then the $G$-action on $A$ is free in the sense of Ellwood~\cite{da-e00}. Hence by the results of Baum, de Commer, and Hajac~\cite{bdch-17}, we have that at the algebraic level the triple $(A,A^G, G)$ is a compact quantum principal bundle.
\item[(iv)] There are many examples of actions with finite local-triviality dimensions including antipodal actions on quantum spheres and gauge actions on various graph C*-algebras.
\end{enumerate}

The following definition is inspired by Theorem~\ref{numprin} and Definition~\ref{ltdim}.
\begin{definition}\label{fulllt}
Let $G$ be a locally compact Hausdorff group acting on a C*-algebra $A$.
\begin{enumerate}
\item We say that $A$ is a {\em weakly locally trivial $G$-C*-algebra} if there exist unital $G$-equivariant \mbox{$\kappa$-con}\-tin\-u\-ous $*$-homomorphisms $\rho_i:C(\cC G)\to \Gamma(K_A)$, $i\in\bN$, such that the sequence $(\sum_{i=0}^N\rho_i({\rm t}))_N$ 
converges to an invertible element $x\in \Gamma(K_A)$ in the $\kappa$-topology.
\item We say that $A$ is a {\em locally trivial $G$-C*-algebra} if there exist unital \mbox{$G$-equi}\-variant \mbox{$\kappa$-con}\-tin\-u\-ous \mbox{$*$-homo}\-morphisms $\rho_i:C(\cC G)\to \Gamma(K_A)$, $i\in\bN$, such that the sequence ($\sum_{i=0}^N\rho_i({\rm t}))_N$
converges to $1$ in the $\kappa$-topology.
\item We say that $A$ is a {\em strongly locally trivial $G$-C*-algebra} if there exist unital $G$-equivariant \mbox{$\kappa$-con}\-tin\-u\-ous \mbox{$*$-homo}\-morphisms $\rho_i:C(\cC G)\to \Gamma(K_A)$, $i\in\bN$, with commuting images and such that ($\sum_{i=0}^N\rho_i({\rm t}))_N$
converges to $1$ in the $\kappa$-topology.
\end{enumerate}
\end{definition}
Here are some observations and remarks:
\begin{enumerate}
\item[(i)] If $G$ is compact Hausdorff, then $\cC G$ is also compact Hausdorff. This turns $C(\cC G)$ into a~unital C*-algebra for which the $\kappa$-topology coincides with the norm topology. Furthermore, if $A$ is unital, then $\Gamma(K_A)=A$ and again the $\kappa$-topology coincides with the norm topology. Therefore, the unital $G$-equivariant $*$-homomorphisms $\rho_i$ in Definition~\ref{fulllt} are automatically continuous. We obtain that
\begin{align*}
\dim_{\rm WLT}^G(A)<\infty\quad&\iff\quad \text{$A$ is a unital weakly locally trivial $G$-C*-algebra},\\
\dim_{\rm LT}^G(A)<\infty\quad&~~~\Rightarrow~\quad\text{$A$ is a unital locally trivial $G$-C*-algebra},\\
\dim_{\rm SLT}^G(A)<\infty\quad&\iff\quad\text{$A$ is a unital strongly locally trivial $G$-C*-algebra}.
\end{align*}
The implication from left to right for $\dim^G_{\rm LT}$ is not clear at the moment, but we conjecture that it holds.
\item[(ii)] It might happen even beyond the compact/unital case that the number of $*$-homomorphisms $\rho_i$ from Definition~\ref{fulllt} is finite and then we could talk about the local-triviality dimension of a given action. However, this time finiteness of this dimension in the commutative case will not recover the notion of a principal bundle. This was possible only for unital C*-algebras generalizing compact spaces, since then every trivilizing cover has a finite subcover.
\item[(iii)]
Note that we have the following implications for $G$-C*-algebras:
\[
\text{strong local triviality}\quad\Rightarrow\quad\text{local triviality}\quad\Rightarrow\quad\text{weak local triviality}.
\]
\item[(iv)] 
Let $G$ be a locally compact Hausdorff group, let $A$ be a weakly locally trivial $G$-C*-algebra, and let $B$ be a $G$-C*-algebra. If there is a strictly non-degenerate $G$-morphism $A\to M(B)$, then it follows from Proposition~\ref{stricnondeg} that $B$ is a weakly locally trivial $G$-C*-algebra.
Hence, if $I$ is a $G$-invariant ideal of $A$, then $A/I$ is a weakly locally trivial $G$-C*-algebra. The same results hold if everywhere above we replace weakly locally trivial with either locally trivial or strongly locally trivial.
\item[(v)] 
The idea of describing the action of a group on a C*-algebra by equivariant maps into a~different object that can be obtained from that C*-algebra is not new. The Rokhlin dimension of an action of a second-countable compact group $G$ on a $\sigma$-unital C*-algebra $A$ is defined using $G$-equivariant completely positive contractive order zero maps from $C(G)$ to the central sequence algebra $A_\infty\cap A'$~(see e.g.~\cite{e-g19}). 
\end{enumerate}


\subsection{Classes of examples of locally trivial $G$-C*-algebras}

Definition~\ref{fulllt} enables us to talk about local triviality of non-unital $G$-C*-algebras. Recall that up to this point even actions of finite groups on non-unital C*-algebras were beyond the scope of applicability of the local-triviality dimensions. Therefore, we obtain a plethora of new examples.

\subsubsection{Classical numerable principal bundles}

First, observe that Definition~\ref{fulllt} generalizes the notion of a locally compact Hausdorff numerable principal $G$-bundle.
\begin{proposition}\label{commult}
Let $G$ be a locally compact Hausdorff group acting on a locally compact Hausdorff space $X$. Then the following are equivalent
\begin{enumerate}
\item[{\rm (1)}] $X\to X/G$ is a numerable principal $G$-bundle,
\item[{\rm (2)}] $C_0(X)$ is a strongly locally trivial $G$-C*-algebra,
\item[{\rm (3)}] $C_0(X)$ is a locally trivial $G$-C*-algebra,
\item[{\rm (4)}] $C_0(X)$ is a weakly locally trivial $G$-C*-algebra.
\end{enumerate}
\end{proposition}
\begin{proof}
(1) $\Rightarrow$ (2). Assume that $X\to X/G$ is a numerable principal bundle. By Theorem~\ref{numprin}, there exist $G$-equivariant continuous maps
\[
\rho_i:X\longrightarrow \cC G,\qquad i\in\bN,
\] 
such that $(\sum_{i=0}^N({\rm t}\circ\rho_i))_{N\in\bN}\subset C(X)$ converges to $1$ in the $\kappa$-topology. Since both $X$ and $\cC G$ are completely regular Hausdorff spaces, we use (\ref{funind}) to obtain unital $G$-equivariant $\kappa$-continuous $*$-homomorphisms
\[
\rho^*_i:C(\cC G)\longrightarrow C(X)\cong\Gamma(K_{C_0(X)}),\qquad i\in\bN,
\]
Since $C(X)$ is commutative and $\rho^*_i({\rm t})={\rm t}\circ\rho_i$\,, the result follows.

The implications (2) $\Rightarrow$ (3) and (3) $\Rightarrow$ (4) are immediate.

(4) $\Rightarrow$ (1). Assume that $C_0(X)$ is a weakly locally trivial $G$-C*-algebra. This means that there are unital $G$-equivariant $\kappa$-continuous $*$-homomorphisms
\[
\rho^*_i:C(\cC G)\longrightarrow\Gamma(K_{C_0(X)})\cong C(X),\qquad i\in\bN,
\]
such that $(\sum_{i=0}^N\rho^*_i({\rm t}))_N$ converges to an invertible element $\sigma\in C(X)$ in the $\kappa$-topology. Since both $X$ and $\cC G$ are completely regular Hausdorff spaces, we use~(\ref{spind}) to obtain $G$-equivariant continuous maps
\[
\rho_i:X\longrightarrow \cC G,\qquad i\in\bN,
\]
such that $(\sum_{i=0}^N({\rm t}\circ\rho_i))_{N\in\bN}$ converges to an invertible element $\sigma\in C(X)$ in the $\kappa$-topology. Since $\sum_{i=0}^N({\rm t}\circ\rho_i)$ is positive for every $N\in\bN$, then $\sigma$ is also positive (see e.g.~\cite[Proposition~8.5]{lt-76} for a~more general statement about Pedersen multiplier algebras). Furthermore, $\sigma(xg)=\sigma(x)$ for all $x\in X$ and $g\in G$. Without loss of generality, we assume that $0< \sigma(x)^{-1}\leq 1$ for all $x\in X$. Consider the following continuous $G$-equivariant maps
\[
l_x:\cC G\longrightarrow\cC G:\qquad [(t,g)]\mapsto [(\sigma(x)^{-1}t,g)],\qquad x\in X,
\]
which can be used to define
\[
\widetilde{\rho}_i:X\longrightarrow \cC G,\qquad \widetilde{\rho}_i(x):=l_x(\rho_i(x)),\qquad i\in\bN.
\]
Each $\widetilde{\rho}_i$ is clearly $G$-equivariant. Since $({\rm t}\circ\widetilde{\rho}_i)(x)=\sigma(x)^{-1}({\rm t}\circ\rho_i)(x)$ and ${\rm g}\circ\widetilde{\rho}_i={\rm g}\circ\rho_i$ (whenever it is well-defined), we conclude from Proposition~\ref{topequiv} that each $\widetilde{\rho}_i$ is continuous. Observe that $(\sum_{i=0}^N({\rm t}\circ\widetilde{\rho}_i))$ converges to $1$ in the $\kappa$-topology. From Theorem~\ref{numprin} we conclude that $X\to X/G$ is a numerable principal $G$-bundle.
\end{proof}
Since partitions of unity enable us to formulate the notion of a locally trivial $G$-C*-algebra, we do not see how to get rid of the numerability assumption. Recall, however, that only numerable principal bundles can be classified using the Milnor universal space.


\subsubsection{Principal $G$-$C_0(Y)$-algebras}\label{classlt}

In this subsection, we consider a semiclassical generalization of principal bundles to the world of noncommutative topology. The next definition is a natural one and it was probably considered by other authors, but we include it for completeness.

\begin{definition}\label{centerprin}
Let $G$ be a locally compact Hausdorff group acting on a C*-algebra $A$. We say that $A$ is a {\em principal $G$-$C_0(Y)$-algebra} if there is a $G$-equivariant $*$-homomorphism 
\[
\chi: C_0(Y)\longrightarrow ZM(A)
\]
for some locally compact Hausdorff principal $G$-bundle $Y\to Y/G$, such that $\chi(C_0(Y))A$ is norm-dense in $A$. If the bundle $Y\to Y/G$ is numerable, we say that $A$ is a {\em numerable principal $G$-$C_0(Y)$-algebra}.
\end{definition}
In the above definition, if $Y$ is second-countable, then $A$ is numerable. If $G$ is additionally second-countable and torsion-free, then Definition~\ref{centerprin} is equivalent to the definition of a proper $G$-C*-algebra of Guentner, Higson, and Trout~\cite[Definition~8.2]{ght-00}, which generalizes the notion of a~proper action in the sense of Baum, Connes, and Higson~\cite[Definition~1.3]{bch-94}. 

Let $\alpha:G\to{\rm Aut}(A)$ be a~continuous action of a~locally compact Hausdorff group on $A$. In what follows, we assume that $\widehat{A}$ is Hausdorff. The continuous action of $G$ on $\widehat{A}$ given by~(\ref{specact}) induces (via~(\ref{calgact})) an action $\beta:G\to {\rm Aut}(C_0(\widehat{A}\,))$ such that
\begin{equation}\label{equidh}
\alpha_g(f\cdot a)=\beta_g(f)\cdot\alpha_g(a),\qquad g\in G,\quad f\in C_0(\widehat{A}\,),\quad a\in A,
\end{equation}
(see e.g.~\cite[Lemma~7.1]{rw-98}). Here $f\cdot a$ is the element given by the Dauns--Hofmann theorem (see e.g.~\cite[Theorem~3]{eo-74}). Define the following map
\[
\chi:C_0(\widehat{A}\,)\longrightarrow ZM(A),\qquad \chi(f)a:=f\cdot a.
\]
It is evident that $\chi$ is a non-degenerate $*$-homomorphism. Furthermore, the equation~(\ref{equidh}) implies that it is $G$-equivariant. Therefore, if $\widehat{A}\to\widehat{A}/G$ is a (numerable) principal $G$-bundle, then $A$ is a~(numerable) principal  $G$-$C_0(\widehat{A}\,)$-algebra. Next, assume that $A$ is a (numerable) principal $G$-$C_0(Y)$-algebra with Hausdorff spectrum $\widehat{A}$, so there is a $G$-equivariant non-degenerate $*$-homomorphism $\chi:C_0(Y)\to ZM(A)\cong C_b(\widehat{A}\,)$. From the anti-equivalence~(\ref{GWor}), we obtain a continuous $G$-map $\widehat{A}\to Y$. Since $Y$ is a~(numerable) principal $G$-bundle, then so is $\widehat{A}$ (it suffices to pull back the trivializing cover on $Y$ to the spectrum~$\widehat{A}$\,). Therefore, we arrived at the following proposition.

\begin{proposition}\label{claspec}
Let $G$ be a locally compact Hausdorff group acting on a C*-algebra $A$ with Hausdorff spectrum $\widehat{A}$. Then $\widehat{A}\to\widehat{A}/G$ is a (numerable) principal $G$-bundle if and only if $A$ is a~(numerable) principal $G$-$C_0(Y)$-algebra for some locally compact Hausdorff $Y$.\hfill$\blacksquare$
\end{proposition}
As a corollary, we obtain that Definition~\ref{centerprin} indeed generalizes the notion of a~locally compact Hausdorff principal bundle.
\begin{corollary}
Let $G$ be a locally compact Hausdorff group acting on a locally compact Hausdorff space $X$. Then the following are equivalent
\begin{enumerate}
\item[{\rm (1)}] $C_0(X)$ is a (numerable) principal $G$-$C_0(Y)$-algebra for some locally compact Hausdorff $Y$.
\item[{\rm (2)}] $X\to X/G$ is a (numerable) principal $G$-bundle.\hfill$\blacksquare$
\end{enumerate}
\end{corollary}

Next, we prove that every principal numerable $G$-$C_0(Y)$-algebra is locally trivial in the sense of Definition~\ref{fulllt}.
\begin{proposition}\label{numprinloc}
Let $G$ be a locally compact Hausdorff group and let $Y\to Y/G$ be a locally compact Hausdorff numerable principal $G$-bundle. Then every principal numerable $G$-$C_0(Y)$-algebra is a locally trivial $G$-C*-algebra.
\end{proposition}
\begin{proof}
Let $A$ be a principal $G$-$C_0(Y)$-algebra. There is a $G$-equivariant $*$-homomorphism 
\[
\chi:C_0(Y)\longrightarrow ZM(A)
\] 
such that $\chi(C_0(Y))A$ is norm-dense in $A$. We infer from Proposition~\ref{pedcom} and Proposition~\ref{stricnondeg} that $\chi$ extends to a~unital $G$-equivariant $\kappa$-continuous $*$-homomorphism $\widetilde{\chi}:C(Y)\to\Gamma(K_A)$. Furthermore, since $Y\to Y/G$ is a numerable principal $G$-bundle, we know from Proposition~\ref{commult} that there exist unital $G$-equivariant $*$-homomorphisms $\rho_i:C(\cC G)\to \Gamma(K_{C_0(Y)})\cong C(Y)$, $i\in\bN$, such that $(\sum_{i=0}^N\rho_i({\rm t}))_N$ converges to $1$ in the $\kappa$-topology. Define $\widetilde{\rho}_i:=\widetilde{\chi}\circ\rho_i$. Then
\[
\widetilde{\rho}_i:C(\cC G)\longrightarrow \Gamma(K_A),\qquad i\in\bN,
\]
are unital $G$-equivariant $\kappa$-continous $*$-homomorphisms. Finally, observe that
\[
\kappa\lim_{N\to\infty}\sum_{i=0}^N\widetilde{\rho}_i({\rm t})=\kappa\lim_{N\to\infty}\sum_{i=0}^N\widetilde{\chi}(\rho_i({\rm t}))=\widetilde{\chi}\left(\kappa\lim_{N\to\infty}\sum_{i=0}^N\rho_i({\rm t})\right)=\widetilde{\chi}(1)=1.
\]
\end{proof}

Let us end this subsection by discussing a certain class of locally trivial $G$-C*-algebras that was already present in the literature. In~\cite{pr-84} and~\cite{lprs-87} a class of actions (and coactions) for which the induced action on the spectrum gives rise to a principal $G$-bundle was studied. Let $\alpha:G\to {\rm Aut}(A)$ be a continuous action of an abelian locally compact Hausdorff group on a C*-algebra with Hausdorff spectrum. The action $\alpha$ is called {\em pointwise unitary} if for each $[\pi]\in\widehat{A}$ there is a (strongly continuous) unitary representation $U$ of $G$ on the Hilbert space $H_\pi$ associated with the $*$-representation $\pi$ such that
\[
\pi(\alpha_g(a))=U_g\pi(a)U_g^*\qquad g\in G,\qquad a\in A.
\] 
If the above holds, we say that $U$ implements $\alpha$ in the $*$-representation $\pi$. The action $\alpha$ is {\em locally unitary} if for each $[\pi]\in\widehat{A}$ there is a neighborhood $N$ of $[\pi]$ and a strictly continuous map $u:G\to M(A)$ such that, for each $[\rho]\in N$, $\overline{\rho}\circ u$ implements $\alpha$ in the representation $\rho$. Here $\overline{\rho}$ denotes the extension of $\rho$ to $M(A)$.

Next, consider the action $\widehat{\alpha}:\widehat{G}\to{\rm Aut}(A\rtimes G)$ of the dual group $\widehat{G}$ on the crossed product C*-algebra $A\rtimes G$ (see e.g.~\cite[p.~176]{rw-98}). Note that since $G$ is abelian, we do not have to distinguish between the reduced and maximal crossed product. Recall that there is an action of $\widehat{G}$ on $(A\rtimes G)^{\wedge}$ (see e.g.~(\ref{specact})). 
From~\cite[Theorem~5.9]{lprs-87} it follows that if $\alpha$ is locally unitary then $(A\rtimes G)^\wedge$ is Hausdorff and the restriction map $(A\rtimes G)^\wedge\to \widehat{A}$ gives rise to a principal $\widehat{G}$-bundle. This in turn implies, by Proposition~\ref{claspec} and Proposition~\ref{numprinloc}, that $A\rtimes G$ is a locally trivial $\widehat{G}$-C*-algebra whenever $(A\rtimes G)^\wedge\to \widehat{A}$ is numerable.
\begin{remark}
Although both cited theorems from~\cite{lprs-87} are phrased in the language of coactions, our formulation is equivalent due to~\cite[Remark~5.7 (2)]{lprs-87}.
\end{remark}


\subsubsection{Actions of $U(1)$ and its cyclic subgroups}
Let $G$ be a compact Hausdorff abelian group and let $\widehat{G}$ denote its Pontryagin dual. Then a~continuous action $\alpha$ of $G$ on a C*-algebra $A$ induces a $\widehat{G}$-grading as follows
\[
A_\chi:=\{a\in A~:~\alpha_g(a)=\chi(g)a,\;g\in G\},\qquad \chi\in\widehat{G}.
\]

\begin{proposition}\label{u1act}
Let $U(1)$ act on a C*-algebra $A$ and let $\gamma$ denote the generator of $\mathbb{Z}=\widehat{U(1)}$. If there exist normal elements $a_i\in A_\gamma$, $i\in\bN$, 
such that $(\sum_{i=0}^Na^*_ia_i)_{N\in\bN}$ is an approximate unit, then $A$ is a~locally trivial $U(1)$-C*-algebra.
\end{proposition}
\begin{proof}
Assume that there are normal elements $a_i\in A$, $i\in\bN$, satisfying the assumptions of the proposition. Let $u$ be the unitary generator of the C*-algebra $C(U(1))$ and consider the function $\sqrt{\rm t}u\in C(\cC U(1))$ defined by $\sqrt{\rm t}u([(t,z)]):=\sqrt{t}z$.
It is straightforward to show that the assignments
\[
C(\cC U(1))\ni \sqrt{\rm t}u\longmapsto a_i\in A,\qquad i\in\bN,
\]
give rise to unital $*$-homomorphisms
\[
\rho_i:C(\cC U(1))\longrightarrow M(A)\subseteq\Gamma(K_A),\qquad i\in\bN.
\]
The $U(1)$-equivariance of a given $\rho_i$ follows from the fact that $a_i\in A_\gamma$.
Finally, observe that
\[
\left(\sum_{i=0}^N\rho_i({\rm t})\right)_{N\in\bN}=\left(\sum_{i=0}^N\rho_i\left(\left(\sqrt{\rm t}u^*\right)\left(\sqrt{\rm t}u\right)\right)\right)_{N\in\bN}=\left(\sum_{i=0}^Na^*_ia_i\right)_{N\in\bN}.
\]
Since $(\sum_{i=0}^Na^*_ia_i)_{N\in\bN}$ is an approximate unit, it converges to 1 in the $\kappa$-topology.
\end{proof}

Let $n\in\bN$ and let $\mathbb{Z}/n\mathbb{Z}$ be a cyclic group of order $n$. Let
\[
{\rm Star}_n:=\{z\in\bC~:~z^n\in[0,1]\}	
\]
be the star-convex set with center $0$ and generated by the $n$th roots of unity. The proof of the following proposition is analogous to the proof of Proposition~\ref{u1act}.

\begin{proposition}\label{znact}
Let $\mathbb{Z}/n\mathbb{Z}$ act on a C*-algebra $A$ and let $\gamma$ denote the generator of $\mathbb{Z}/n\mathbb{Z}=\widehat{\mathbb{Z}/n\mathbb{Z}}$. If there exist normal elements $a_i\in A_\gamma$, $i\in\bN$, such that the spectrum of each $a_i$ is contained in ${\rm Star}_n$ and 
such that $(\sum_{i=0}^Na^*_ia_i)_{N\in\bN}$ is an approximate unit, then $A$ is a~locally trivial $\mathbb{Z}/n\mathbb{Z}$-C*-algebra.
\end{proposition}
Note that for $n=2$ the above elements $a_i$ are odd and self-adjoint and $(\sum_{i=0}^Na_i^2)_{N\in\bN}$ is an approximate unit. 


\subsubsection{$\mathbb{Z}/2\mathbb{Z}$-actions on non-unital graph C*-algebras}

Recall that a simple C*-algebra cannot be a principal $G$-$C_0(Y)$-algebra. However, we know examples of simple unital locally trivial $G$-C*-algebras, namely one can show that the Cuntz algebra $\mathcal{O}_n$ is a locally trivial $\bZ/2\bZ$-C*-algebra (see e.g.~\cite[Subsection~5.1]{cpt-21}). As far as the author knows, there are no examples of this type for non-unital C*-algebras in the literature. Let us use our new methods to find such an example. To this end, we need to discuss certain actions on graph C*-algebras. We refer the reader to~\cite{bprs00} for a~detailed treatment of the theory.

A~{\em directed graph} is a~quadruple $(E^0,E^1,s,r)$, where $E^0$ is the set of vertices, $E^1$ is the set of edges, and $s,r:E^1\to E^0$ are the source and range maps respectively. We assume that both $E^0$ and $E^1$ are countable. A~graph is called {\em row-finite} if for every $v\in E^0$ we have that $|s^{-1}(v)|<\infty$. By a \emph{path} $e$ of length $|e|=k\geq 1$ we mean a directed path, i.e.\ a sequence of edges $\gamma:=e_1\ldots e_k$, with $r(e_i)=s(e_{i+1})$ for all $i=1,\ldots,k-1$. By convention, vertices are paths of length $0$. Let us denote the set of finite paths by $FP(E)$. One can extend the range and source maps to $FP(E)$ by putting $r(\mu):=r(e_n)$ and $s(\mu):=s(e_1)$ for $\mu=e_1\ldots e_n$ and $r(v)=s(v)=v$ for vertices.

The {\em  \mbox{C*-algebra} $C^*(E)$ of a row-finite graph $E$} is the universal \mbox{C*-algebra} generated by mutually 
orthogonal projections \mbox{$\big\{P_v\,|\,v\in E_0\big\}$} and partial isometries $\big\{S_e\,|\,e\in E_1\big\}$ satisfying the \emph{Cuntz--Krieger relations}:
\begin{align}
S_e^*S_e &=P_{r(e)} && \text{for all }e\in E_1\text{, and}\tag{\text{CK1}} \label{eq:CK1} \\
\sum_{e\in s^{-1}(v)}\!\! S_eS_e^*&=P_v && \text{for all }v\in E_0\text{ with $s^{-1}(v)\neq \emptyset$.}\tag{CK2} \label{eq:CK2}
\end{align}
For a path $\mu=e_1\ldots e_n$, we write $S_\mu:=S_{e_1}\ldots S_{e_n}$. Let us recall some results on graph C*-algebras that will be needed in our examples:
\begin{enumerate}
\item[(i)] In every graph C*-algebra of a row-finite graph we have that $P_{s(e)}S_e=S_e=S_eP_{r(e)}$ and $S^*_fS_e=0$ if and only if $f\neq e$. 
\item[(ii)] The graph C*-algebra $C^*(E)$ is unital if and only if $|E^0|<\infty$. 
\item[(iii)] We have that $C^*(E)\cong\overline{\rm span}\{S_\mu S_\nu^*~:~\mu,\nu\in FP(E),~r(\mu)=r(\nu)\}$.
\item[(iv)] It is well known that if $X\subseteq E^0$, then $\sum_{v\in X}P_v$ converges strictly in the multiplier algebra $M(C^*(E))$ to a projection $P_X$~\cite[Lemma~1.1]{bprs00}. If $X=E^0$, then $\sum_{v\in E^0}P_v$ converges strictly to the unit in $M(C^*(E))$.
\item[(v)] Every graph C*-algebra comes equipped with the {\em gauge $U(1)$-action} $\alpha$ defined on the generators as follows
\[
\alpha_z(P_v):=P_v,\qquad\alpha_z(S_e)=zS_e\,,\qquad z\in U(1),\quad v\in E^0,\quad e\in E^1. 
\]
\end{enumerate}

Here we are only interested in the restriction of the gauge action to $\mathbb{Z}/2\mathbb{Z}$. To find locally trivial $\bZ/2\bZ$-C*-algebras among graph C*-algebras, one can use Proposition~\ref{znact} and the fact that the generalized sum of vertex projections strictly converges to $1$. Next, we present the aforementioned example of a simple non-unital locally trivial $\mathbb{Z}/2\mathbb{Z}$-C*-algebra. 
\begin{example}
The C*-algebra of compact operators $\cK(H)$ on a seprarable Hilbert space $H$ can be realized as the graph C*-algebra (see Figure 1). Recall that $\cK(H)$ is a simple C*-algebra, so it cannot be a principal $\mathbb{Z}/2\bZ$-$C_0(Y)$-algebra. However, we will show that it is a locally trivial $\bZ/2\bZ$-C*-algebra. This is a~purely noncommutative phenomenon. 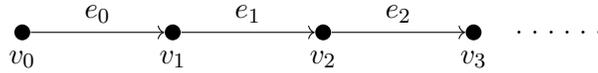
\begin{figure}[h]
\centering
\begin{tikzpicture}
\tikzstyle{vertex}=[circle,fill=black,minimum size=6pt,inner sep=0pt]
\tikzstyle{edge}=[draw,->]
\tikzstyle{loop}=[draw,->,min distance=20mm,in=130,out=50,looseness=50]
\node[vertex,label=below:$v_0$] (a) at  (0,0) {};
\node[vertex,label=below:$v_1$] (b) at (2, 0) {};
\node[vertex,label=below:$v_2$] (c) at (4, 0) {};
\node[vertex,label=below:$v_3$] (d) at (6, 0) {};
\node (e) at (6.5,0) {~~~~~~~~~~.\;.\;.\;.\;.\;.};
\path (a) edge[edge] node[above] {$e_{0}$} (b);
\path (b) edge[edge] node[above] {$e_{1}$} (c);
\path (c) edge[edge] node[above] {$e_{2}$} (d);
\end{tikzpicture}
\caption{Graph representation of $\cK$.}
\end{figure}

Consider the restriction of the gauge action on $\cK(H)$ to $\mathbb{Z}/2\mathbb{Z}$. For brevity, we use the following notation $P_n:=P_{v_n}$ and $S_n:=S_{e_n}$ for all $n\in\bN$. Consider the following odd self-adjoint elements
\[
x_{n}:=S_{2n}+S_{2n}^*\,, \qquad n\in\bN.
\]
Observe that
\[
\sum_{n=0}^{N}x_n^2=\sum_{n=0}^N(P_{2n}+P_{2n+1})=\sum_{n=0}^{2N+1} P_n\,.
\]
Therefore, $(\sum_{n=0}^{N}x_n^2)_{N\in\bN}$ is an approximate unit in $\cK$. By Proposition~\ref{znact}, we obtain that $\cK$ is a~locally trivial $\bZ/2\bZ$-C*-algebra.
\end{example}

\begin{example}
\noindent We define a graph $E_\infty$ as follows (see Figure~2): 
\[
E^0_\infty:=\{v_n~:~n\in\bN\},\qquad E^1_\infty:=\{e_n\,,e_{n,n+1}~:~n\in\bN\}, 
\]
where $s(e_n)=r(e_n)=v_n$, $s(e_{n,n+1})=v_n$, and $r(e_{n,n+1})=v_{n+1}$ for all $n\in\bN$.

\begin{figure}[h!]
\centering
\begin{tikzpicture}\label{infsphere}
\tikzstyle{vertex}=[circle,fill=black,minimum size=6pt,inner sep=0pt]
\tikzstyle{edge}=[draw,->]
\tikzstyle{loop}=[draw,->,min distance=20mm,in=130,out=50,looseness=50]
 vertices
\node[vertex,label=below:$v_1$] (a) at  (0,0) {};
\node[vertex,label=below:$v_2$] (b) at (2, 0) {};
\node[vertex,label=below:$v_3$] (c) at (4, 0) {};
\node[vertex,label=below:$v_4$] (d) at (6, 0) {};
\node (e) at (6.5,0) {~~~~~~~~~~.\;.\;.};
edges
\path (a) edge [edge, in = 45, out = 135, looseness = 30] node[above] {$e_{0}$} (a);
\path (a) edge[edge] node[above] {$e_{01}$} (b);
\path (b) edge [edge, in = 45, out = 135, looseness = 30] node[above] {$e_{1}$} (b);
\path (b) edge[edge] node[above] {$e_{12}$} (c);
\path (c) edge [edge, in = 45, out = 135, looseness = 30] node[above] {$e_{2}$} (c);
\path (c) edge[edge] node[above] {$e_{23}$} (d);
\path (d) edge [edge, in = 45, out = 135, looseness = 30] node[above] {$e_{3}$} (d);
\end{tikzpicture}
\caption{The graph $E_\infty$.}
\end{figure}
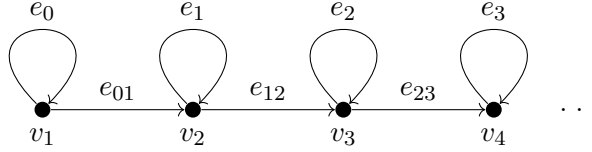
Consider the restriction of the gauge action $\alpha$ on $C^*(E_\infty)$ to the subgroup $\mathbb{Z}/2\mathbb{Z}$. The graph C*-algebra $C^*(E)$ is non-unital, so we need to use our new methods to establish whether $C^*(E_\infty)$ is a locally trivial $\mathbb{Z}/2\mathbb{Z}$-C*-algebra. By Proposition~\ref{znact}, we need to find odd self-adjoint elements $x_i$, $i\in\bN$, such that $(\sum_{i=0}^N x^2_i)_{N\in\bN}$ converges strictly to the unit in $M(C^*(E_\infty))$. We put
\[
\chi_n:=\frac{1}{3}\left(1+\frac{(-1)^n}{2^{n+1}}\right),\qquad n\in\mathbb{N}.
\]
Observe that $\chi_0=\frac{1}{2}$, $2\chi_n+\chi_{n-1}=1$ for $n\in\mathbb{N}$, and $\lim_{n\to\infty}\chi_n=\frac{1}{3}$. Consider the following elements
\[
y_n:=\sqrt{2\chi_n}(S_n+S_{n,n+1}),
\]
where $S_n:=S_{e_n}$ and $S_{n,n+1}:=S_{e_{n,n+1}}$. We claim that the real and imaginary parts of the elements $y_n$ give us the needed odd self-adjoint elements, i.e. let
\[
x_{2n}:=\frac{y_n+y_n^*}{2},\qquad x_{2n+1}:=\frac{y_n-y_n^*}{2i},\qquad n\in\bN.
\]
Note that
\[
\sum_{n=0}^{2N+1}x^2_n=\sum_{n=0}^N\left(x_{2n}^2+x_{2n+1}^2\right)=\sum_{n=0}^N\chi_n(2 P_n+P_{n+1})=\sum_{n=0}^NP_n+\chi_NP_{N+1},
\]
where $P_n:=P_{v_n}$. Note that for any finite path $\mu$ in $E_\infty$, we can find $N\in\mathbb{N}$ such that for every $n>N$ we have that $P_nS_\mu=0$. Since every $a\in C^*(E_\infty)$ can be approximated by a linear combination of the elements of the form $S_\mu S_\nu^*$, where $\mu,\nu$ are finite paths in $E_\infty$ of the same range, we conclude that $\lim_{N\to\infty}P_Na=0$ for all $a\in C^*(E_\infty)$. Therefore, for any $a\in C^*(E_\infty)$, we have that
\begin{align*}
\left\|\sum_{n=0}^{2N+1}x_n^2a-a\right\|&=\left\|\sum_{n=0}^{2N+1}x_n^2a-\sum_{n=0}^NP_na+\sum_{n=0}^NP_na-a\right\|\\
&\leq \|\chi_NP_{N+1}a\|+\left\|\sum_{n=0}^NP_na-a\right\|\longrightarrow 0.
\end{align*}
This implies that $(\sum_{n=1}^{2N}x^2_n)_{N\in\bN}$ converges strictly to the unit in $M(C^*(E_\infty))$. A similar calculation shows that $(\sum_{n=1}^{2N+1}x^2_n)_{N\in\bN}$ is also an approximate unit, which implies that $C^*(E_\infty)$ is a locally trivial $\mathbb{Z}/2\mathbb{Z}$-C*-algebra.
\end{example}


\section{Relation to free actions}\label{setthemfree}

If $X\to X/G$ is a principal $G$-bundle, then the action of $G$ on $X$ is automatically free. In this subsection we show that in the noncommutative case local triviality also implies a version of freeness of the action. There are many generalizations of freeness to the setting of actions of locally compact Hausdorff groups on C*-algebras (see e.g.~\cite{nc-p09} for the case of finite group actions).


\subsection{$\kappa$-freeness and the Ellwood condition}

We begin with a characterization of injective maps between locally compact Hausdorff spaces in terms of their $*$-algebras of all continuous functions. The proof we present here is a modification of the proof of~\cite[Lemma~2.7]{da-e00} (see also~\cite{bhms-07}) and heavily relies on the Tietze extension theorem for locally compact Hausdorff spaces, which states that every continuous function on a~non-empty compact subset of a locally compact Hausdorff space can be extended uniquely to a~compactly supported continuous function on the whole space with the same sup-norm. 
\begin{lemma}\label{kapinj}
Let $Y$ and $Z$ be locally compact Hausdorff spaces and let $f:Y\to Z$ be a~continuous map. Then $f$ is injective if and only if the induced unital $\kappa$-continuous $*$-homomorphism \mbox{$f^*:C(Z)\to C(Y)$} (see~(\ref{funind})) has $\kappa$-dense range.
\end{lemma}
\begin{proof}
($\Rightarrow$) Assume that $f:Y\to Z$ is injective. Since $C_b(Y)$ is $\kappa$-dense in $C(Y)$, it suffices to show that any element in $C_b(Y)$ can be approximated by the elements in the range of $f^*$. Let $\omega\in C_b(Y)$ and let $\emptyset\neq K\subset Y$ be compact. From the Tietze extension theorem it follows that there is $c_K\in C_c(Y)$ such that $c_K(K)\subseteq \{1\}$ and $\|c_K\|=1$. Take $\omega_K:=\omega\, c_K\in C(K)$. By the injectivity of $f$, for any compact subset $K\subset Y$, there is a $*$-isomorphism of C*-algebras $\Psi_K: C(f(K))\to C(K)$. Again using the Tietze extension theorem, the function $\Psi_K^{-1}(\omega_K)\in C(f(K))$ can be extended to the~function $\eta_K\in C_b(Z)$ such that $\|\Psi_K^{-1}(\omega_K)\|=\|\eta_K\|$. Observe that
\begin{align*}
\|f^*(\eta_K)-\omega\|&=\sup_{y\in Y}|f^*(\eta_K)(y)-\omega(y)|=\sup_{y\in Y}|\eta_K(f(y))-\omega(y)|\\
&\leq\sup_{z\in Z}|\eta_K(z)|+\sup_{y\in Y}|\omega(y)|\leq \|\eta_K\|+\|\omega\|\\
&\leq \|\omega_K\|+\|\omega\|\leq \|\omega\|\|c_K\|+\|\omega\|=2\|\omega\|.
\end{align*}
Finally, for any $k\in C_c(Y)$ with ${\rm supp}\,k=L$, we have that
\begin{align*}
p_L(f^*(\eta_K)-\omega)&=\sup_{y\in Y}|(\eta_K(f(y))-\omega(y))k(y)|=\sup_{y\notin K}|\eta_K(f(y))-\omega(y)||k(y)|\\
&\leq \|f^*(\eta_K)-\omega\|\sup_{y\notin K}|k(y)|\leq 2\|\omega\|\sup_{y\notin K}|k(y)|.
\end{align*}
Hence for any $K$ such that $L\subseteq K$, we have that $p_L(f^*(\eta_K)-\omega)=0$. This implies that $(f^*(\eta_K))_{\emptyset\neq K\subset Y}$ converges to $\omega$ in the $\kappa$-topology.

($\Leftarrow$) Assume that $f$ is not injective. Then there are points $y_1\neq y_2$ such that $f(y_1)=f(y_2)$. Since $\{y_1\,,y_2\}$ is a compact subspace, then by the Tietze extension theorem there is a function $\omega\in C_c(Y)$ such that $\omega(y_1)=\omega(y_2)=1$. Continuous functions on a locally compact Hausdorff space $Y$ separate points, hence, there is also a function $\psi\in C(Y)$ such that $\psi(y_1)\neq\psi(y_2)$. If the range of $f^*$ would be $\kappa$-dense in $C(Y)$, then there would exist a net $(\eta_\lambda)_{\lambda\in\Lambda}$ in $f^*(C(Z))$ such that
\[
\sup_{y\in K}|\eta_\lambda(y)-\psi(y)\omega(y)|\longrightarrow 0,
\]
for all compact subsets $K\subset Y$. However, for $K=\{y_1,y_2\}$ we would have that
\begin{align*}
\max&\{|\eta_\lambda(y_1)-\psi(y_1)\omega(y_1)|,|\eta_\lambda(y_2)-\psi(y_2)\omega(y_2)|\}\\
&\geq\frac{1}{2}\left(|\eta_\lambda(y_1)-\psi(y_1)\omega(y_1)|+|\eta_\lambda(y_2)-\psi(y_2)\omega(y_2)|\right)\\
&\geq \frac{|\psi(y_1)\omega(y_1)-\psi(y_2)\omega(y_2)|}{2}>0,
\end{align*}
where in the last step we used the triangle inequality and the fact that $\eta_\lambda(y_1)=\eta_\lambda(y_2)$ for all $\lambda\in\Lambda$. Therefore, we arrive at a contradiction.
\end{proof}

To proceed, we need some results on the spaces of continuous functions with values in unital locally convex Hausdorff $*$-algebras, i.e. locally convex Hausdorff spaces equipped with a compatible unital \mbox{$*$-al}\-gebra structure. Let $X$ be a locally compact Hausdorff space and let $V$ be a~unital locally convex Hausdorff $*$-algebra with the family of seminorms $\{s_\lambda\}_{\lambda\in\Lambda}$ generating its topology. Let $C(X,V)$ denote the space of all continuous functions $X\to V$. The $\kappa$-topology on $C(X,V)$ is defined using the seminorms
\begin{equation}\label{kapfun}
u\longmapsto \sup_{x\in K}s_\lambda(u(x)),
\end{equation}
where $K\subset X$ is compact. There is also a natural unital $*$-algebra structure on $C(X,V)$ coming from the unital $*$-algebra structure on $V$.
Let us state a generalization of the Stone--Weierstrass theorem that can be deduced from~\cite[Corollary~8]{v-t05}.
\begin{theorem}\label{newstone}
Let $X$ be a Hausdorff space, let $V$ be a~locally convex Hausdorff space, and let $W$ be a vector subspace of $C(X,V)$. If $W$ separates points, $C(X)\cdot W\subset W$, and $W(x):=\{w(x):w\in W\}$ is dense in $V$ for every $x\in X$, then $W$ is $\kappa$-dense in $C(X,V)$.\hfill$\blacksquare$
\end{theorem}
Consider the functions
\begin{equation}\label{fvnct}
(fv)(x):=f(x)v,\qquad f\in C(X),\quad v\in V.
\end{equation}
Then from Theorem~\ref{newstone} it follows that
\begin{equation}\label{splitdense}
{\rm span}\{fv~:~f\in C(X),\; v\in V\}\quad\text{is $\kappa$-dense in}\quad C(X,V).
\end{equation}

The following lemma should be well known, but we present its proof for completeness.
\begin{lemma}\label{isoprod}
Let $Y$ and $Z$ be locally compact Hausdorff spaces. There are isomorphisms of topological $*$-algebras
\[
C(Y,C(Z))\cong C(Y\times Z)\cong C(Z,C(Y)).
\]
Here we consider the compact-open topology on $C(Y\times Z)$ and the topology given by seminorms~(\ref{kapfun}) for the other two $*$-algebras.
\end{lemma}
\begin{proof}
Consider the following assignment
\[
\Phi:C(Y,C(Z))\longrightarrow C(Y\times Z),\qquad u\longmapsto (\,(y,z)\longmapsto (u(y))(z)\,).
\]
Since $Z$ is locally compact Hausdorff, the above map is well defined (see e.g.~\cite[Theorem~46.11]{jr-m00}). It is straightforward to check that $\Phi$ is a unital $\kappa$-continuous $*$-homomorphism. We define a map
\[
\Phi^{-1}:C(Y\times Z)\longrightarrow C(Y,C(Z)),\qquad \xi\longmapsto (\, y\longmapsto \xi(y,\cdot)\,).
\]
It is clear that $\Phi^{-1}$ is a unital $\kappa$-continuous $*$-homomorphism and that $\Phi$ and $\Phi^{-1}$ are mutually inverse. The proof of the other isomorphism is analogous.
\end{proof}

\begin{theorem}\label{frecom}
Let $G$ be a locally compact Hausdorff group acting on a locally compact Hausdorff space $X$. The $G$-action on $X$ is free if and only if
\[
{\rm span}\left\{f_1\alpha_{(\cdot)}(f_2)~:~f_1,f_2\in C_c(X)\right\}\quad\text{is $\kappa$-dense in}\quad C(G,C(X)),
\]
where $\alpha$ is the induced action of $G$ on $C_0(X)$. Here $(f_1\alpha_{(\cdot)}(f_2))(g):=f_1\alpha_{g}(f_2)$.
\end{theorem}
\begin{proof}
First, we need some preparation. The action of $G$ on $X$ is free if and only if the following map
\[
F:X\times G\longrightarrow X\times X:\quad (x,g)\longmapsto (x,xg)
\]
is injective. By Lemma~\ref{kapinj} and Lemma~\ref{isoprod} it is also equivalent to the fact that the following unital $\kappa$-continuous $*$-homomorphism
\[
F^*:C(X\times X)\longrightarrow C(G,C(X)),\qquad (F^*(f)(g))(x):=f(x,xg),
\]
has $\kappa$-dense range. Observe that
\[
W:={\rm span}\{(f_1\times f_2)~:~f_1\,,f_2\in C_c(X)\},
\]
where $(f_1\times f_2)(x_1\,,x_2):=f_1(x_1)f_2(x_2)$, is $\kappa$-dense in $C(X\times X)$. We use the notation
\[
\mathcal{E}:={\rm span}\left\{f_1\alpha_{(\cdot)}(f_2)~:~f_1,f_2\in C_c(X)\right\}.
\]
We infer that $\mathcal{E}=F^*(W)$.

\noindent($\Rightarrow$) Assume that the $G$-action on $X$ is free, which is tantamount to the range of $F^*$ being $\kappa$-dense. Since $\mathcal{E}=F^*(W)$, it is $\kappa$-dense in $F^*(C(X\times X))$. In turn, this implies that $\mathcal{E}$ is $\kappa$-dense in $C(G,C(X))$.

\noindent($\Leftarrow$) Since $\mathcal{E}=F^*(W)\subset  F^*(C(X\times X))$, the $\kappa$-density of $\mathcal{E}$ implies the $\kappa$-density of the range of $F^*$, which is equivalent to the freeness of the action.
\end{proof}

In~\cite{da-e00}, Ellwood introduced a notion of a free action of a locally compact quantum group on \mbox{a~C*-al}\-gebra, which we now specify to actions of classical groups. Let $\alpha:G\to {\rm Aut}(A)$ be a~continuous action of a locally compact Hausdorff group on a C*-algebra. We say that $\alpha$ satisfies the {\em Ellwood condition} if
\[
{\rm span}\{a\alpha_{(\cdot)}(b)~:~a,b\in A\}\quad\text{is strictly dense in}\quad M(C_0(G,A)).
\]
Here $(a\alpha_{(\cdot)}(b))(g):=a\alpha_g(b)$. Under additional assumptions the Ellwood condition is equivalent to other versions of freeness known from the literature.
If $G$ is compact Hausdorff and $A$ is unital, then the Ellwood condition is equivalent to the saturation property introduced by Rieffel~\cite[Definition~1.6]{ma-r90}. The proof of the equivalence of the two conditions in a more general case of compact quantum group actions on unital C*-algebras can be found in~\cite[Theorem~1.1]{dcy-13}. The Ellwood condition is also equivalent to the algebraic Peter--Weyl--Galois condition known from the Hopf--Galois theory (note that the Hopf $*$-algebra in question have to be cosemisimple to correspond to a dense Hopf $*$-subalgebra of the C*-algebra of a compact quantum group). We refer the reader to~\cite{bdch-17} for a~proof of the latter equivalence.

A canonical example of an action satisfying the Ellwood condition is the $G$-action $\Delta$ on $C_0(G)$ induced by the (right) multiplication (see~\cite[Proposition~2.1]{da-e00} for a more general result). It is known that when $A=C_0(X)$, then $\alpha$ satisfies the Ellwood condition if and only if the induced action of $G$ on $X$ is free~\cite[Theorem~2.9(a)]{da-e00}. 

The above discussion and Theorem~\ref{frecom} prompt the following definition.
\begin{definition}
Let $\alpha:G\to{\rm Aut}(A)$ be a continuous action of a locally compact Hausdorff group on a C*-algebra. We say that $\alpha$ is {\em $\kappa$-free} if
\[
{\rm span}\{k_1\alpha_{(\cdot)}(k_2)~:~k_1,k_2\in K_A\}\qquad\text{is $\kappa$-dense in}\qquad C(G,\Gamma(K_A)).
\]
Here $(k_1\alpha_{(\cdot)}(k_2))(g):=k_1\alpha_g(k_2)$.
\end{definition}

The next result shows that the Ellwood condition always implies $\kappa$-freeness. In general, it is not clear whether $\kappa$-freeness implies the Ellwood condition, however, it is true for a~big class of examples.
\begin{proposition}
Let $\alpha:G\to {\rm Aut}(A)$ be a continuous action of a locally compact Hausdorff group on a C*-algebra. If $\alpha$ satisfies the Ellwood condition, then it is $\kappa$-free. The converse holds whenever $A$ is commutative or when $G$ is compact and $A$ is unital.
\end{proposition}
\begin{proof}
Since $\alpha$ satisfies the Ellwood condition and strict density implies $\kappa$-density, we have that
\[
{\rm span}\{a\alpha_{(\cdot)}(b)~:~a,b\in A\}\quad\text{is $\kappa$-dense in}\quad M(C_0(G,A)).
\]
Furthermore, from~\cite[Corollary~3.4]{apt-73}, we know that $M(C_0(G,A))=C_b(G,M(A))$ and the \mbox{$\kappa$-topol}\-ogy on $C_b(G,M(A))$ is given by the restriction of seminorms of the form~(\ref{kapfun}). The result follows by noticing that
\[
{\rm span}\{k_1\alpha_{(\cdot)}(k_2)~:~k_1,k_2\in K_A\}
\]
is $\kappa$-dense in 
\[
{\rm span}\{a\alpha_{(\cdot)}(b)~:~a,b\in A\}
\] 
and that $C_b(G,M(A))$ is $\kappa$-dense in $C(G,\Gamma(K_A))$.

If $A=C_0(X)$ is commutative, Theorem~\ref{frecom} ensures that $\kappa$-freeness of $\alpha$ is equivalent to freeness of the corresponding $G$-action on $X$, which is equivalent to the Ellwood condition by~\cite[Theorem~2.9(a)]{da-e00}. Next, if $G$ is compact and $A$ is unital, then $C(G,\Gamma(K_A))=C(G,A)$. The result follows from the fact that on $C(G,A)$ the strict topology and the $\kappa$-topology coincide.
\end{proof}


\subsection{Local triviality implies $\kappa$-freeness}

In this subsection, we prove that if $A$ is a weakly locally trivial $G$-C*-algebra then the considered $G$-action on $A$ is $\kappa$-free. First, we need the following notation. Assume that there is a~unital $G$-equivariant $\kappa$-continuous $*$-homomorphism
\[
\rho:C(\cC G)\longrightarrow \Gamma(K_A).
\]
We define the following unital $G$-equivariant $\kappa$-continuous $*$-homomorphism
\begin{equation}\label{tilderho}
\widetilde{\rho}:C(G,C(\cC G))\longrightarrow C(G,\Gamma(K_A)),\qquad \widetilde{\rho}(u)(g):=\rho(u(g)).
\end{equation}

\begin{theorem}
Let $\alpha:G\to {\rm Aut}(A)$ be a continuous action of a locally compact Hausdorff group on a C*-algebra $A$. If $A$ is a weakly locally trivial $G$-C*-algebra, then $\alpha$ is $\kappa$-free.
\end{theorem}
\begin{proof}
Throughout the proof we will use the notation 
\[
\mathcal{E}:={\rm span}\{k_1\alpha_{(\cdot)}(k_2)~:~k_1,k_2\in K_A\}.
\]
We have to prove that $\mathcal{E}$ is $\kappa$-dense in $C(G,\Gamma(K_A))$. Note that by~(\ref{splitdense}), we only have to show that the functions $fv$ (see~(\ref{fvnct})), where $f\in C(G)$ and $v\in \Gamma(K_A)$, can be approximated in the $\kappa$-topology by the elements of $\mathcal{E}$. Furthermore, since 
\[
fv=(1_{C(G)}v)(f1_{\Gamma(K_A)})\qquad\text{and}\qquad (1_{C(G)}v)\mathcal{E}\subset\mathcal{E},
\] 
it suffices to approximate functions of the form $f1_{\Gamma(K_A)}$. 

Since $A$ is a weakly locally trivial $G$-C*-algebra, there exist unital $G$-equivariant $\kappa$-continuous $*$-homomorphisms $\rho_i:C(\cC G)\to\Gamma(K_A)$, $i\in\bN$, such that 
\[
\sum_{i=0}^N\rho_i({\rm t})\quad\underset{N\to\infty}{\longrightarrow}\quad x\qquad\text{in the $\kappa$-topology},
\]
for some invertible $x\in \Gamma(K_A)$. The above implies that
\[
fx=\kappa\lim_{N\to\infty} \sum_{i=0}^Nf\rho_i({\rm t})=\kappa\lim_{N\to\infty}\sum_{i=0}^N\widetilde{\rho}_i(f{\rm t}),
\]
where each $\widetilde{\rho}_i$ is given by the formula~(\ref{tilderho}) and $f{\rm t}\in C(G,C(\cC G))$ is defined by~(\ref{fvnct}).
Next, let
\[
\mathcal{G}:={\rm span}\{f_1\Delta_{(\cdot)}(f_2)~:~f_1,f_2\in C_c(G)\}.
\]
Since the $G$-action on $C_0(G)$ satisfies the Ellwood condition, it is also $\kappa$-free. Hence, there is a net $(\gamma_\lambda)\subset\mathcal{G}$ such that
\[
f1_{C(G)}=\kappa\lim_\lambda\gamma_\lambda
\]
in $C(G,C(G))$. 
This in turn implies that
\[
f{\rm t}=\kappa\lim_\lambda {\rm t}\gamma_\lambda
\]
in $C(G,C(\cC G))$, where $({\rm t}\gamma_\lambda)(g):={\rm t}\gamma_\lambda(g)$, $g\in G$, and 
\[
{\rm t}\gamma_\lambda(g):\cC G\longrightarrow\bC:\qquad {\rm t}\gamma_\lambda(g)([(t,g')]):=\begin{cases}0~~& t=0,\\ t(\gamma_\lambda(g))(g')~~&t\neq 0.\end{cases}
\]
Since $\gamma_\lambda\in\mathcal{G}$ for every $\lambda$, there exist $m_\lambda\in\bN$ and $a^\lambda_j,b^\lambda_j\in C_c(G)$, with $j=0,1,\ldots,m_\lambda$, such that
\[
\gamma_\lambda=\sum_{j=0}^{m_\lambda}a^\lambda_j\Delta_{(\cdot)}(b^\lambda_j).
\]
Subsequently, since each $\widetilde{\rho}_i$ is a $\kappa$-continuous $*$-homomorphism and each $\rho_i$ is $G$-equivariant, we obtain that
\[
\sum_{i=0}^N\widetilde{\rho}_i(f{\rm t})=\kappa\lim_\lambda\sum_{i=0}^N\sum_{j=0}^{m_\lambda}\rho_i\left(\sqrt{\rm t}a^\lambda_j\right)\Gamma(\alpha)_{(\cdot)}\left(\rho_i\left(\sqrt{\rm t}b^\lambda_j\right)
\right).
\]
Since $\rho_i\left(\sqrt{\rm t}a^\lambda_j\right)$ and $\rho_i\left(\sqrt{\rm t}b^\lambda_j\right)$ belong to $K_A$ for all $\lambda$, this implies that the functions of the form $fx$ can be approximated in the $\kappa$-topology by the elements of $\mathcal{E}$. By invertibility of $x\in\Gamma(K_A)$, the same conclusion follows for the functions of the form $f1_{\Gamma(K_A)}$.
\end{proof}


\section*{Acknowledgments}\noindent

It is a pleasure to thank Piotr M.~Hajac for his continuous support and for many conversations about (quantum) principal bundles in the non-compact setting. Discussions with Alexandru Chirvasitu on local triviality, pro-C*-algebras, and classifying spaces and with Tomasz Maszczyk on unreduced cones of C*-algebras and the multiplier algebra of the Pedersen ideal, have been a~great help at the beginning of this project. I would also like to thank Tyrone Cutler for an exhaustive answer to the question about the topology of the unreduced cones and joins that I posted on Mathematics Stack Exchange. This work was financed by NCN grant 2020/36/C/ST1/00082 entitled {\it Noncommutative universal spaces for groups and quantum groups}.




\Addresses

\end{document}